\documentclass[11pt]{amsart}
\usepackage{amsmath,amssymb,amsthm,enumitem,xcolor,tikz,subcaption,hyperref}

\usetikzlibrary{decorations.pathreplacing,calligraphy}
\numberwithin{equation}{section}
\hypersetup{  pdfborder={0 0 0},
  colorlinks   = true,
  urlcolor     = blue,
  linkcolor    = blue,
  citecolor   = red}

\newcommand{\N}{\mathbb{N}}
\newcommand{\Z}{\mathbb{Z}}

\newcommand{\cT}{\mathcal{T}}

\renewcommand{\aa}{\underline{a}}
\newcommand{\dd}{\underline{d}}
\renewcommand{\d}{\mathrm{d}}

\DeclareMathOperator{\st}{Stab}
\DeclareMathOperator{\fix}{Fix}
\DeclareMathOperator{\ax}{Axis}

\newtheorem{thm}{Theorem}[section]
\newtheorem{prop}[thm]{Proposition}
\newtheorem{lemma}[thm]{Lemma}
\newtheorem{cor}[thm]{Corollary}

\theoremstyle{definition}
\newtheorem{defn}[thm]{Definition}
\newtheorem{question}[thm]{Question}

\theoremstyle{remark}
\newtheorem{ex}[thm]{Example}
\newtheorem{rem}[thm]{Remark}

\newtheorem{notation}[thm]{Notation}

\title[Right-angled Artin subgroups in one-relator groups]{Right-angled Artin subgroups and free products in one-relator groups}

\author{Ashot Minasyan}
\address[A.\ Minasyan]{CGTA, School of Mathematical Sciences, University of Southampton, Highfield, Southampton, SO17~1BJ, United Kingdom}
\email{aminasyan@gmail.com}
\author{Motiejus Valiunas}
\address[M.\ Valiunas]{Instytut Matematyczny, Uniwersytet Wroc{\l}awski, plac Grunwaldzki 2, 50-384 Wroc{\l}aw, Poland}\thanks{The second author was partially supported by the National Science Centre (Poland) grant 2022/47/D/ST1/00779.}
\email{Motiejus.Valiunas@math.uni.wroc.pl}

\keywords{One-relator groups,  right-angled Artin groups, trace monoids, property $P_{nai}$, $C^*$-simplicity}
\subjclass{20F05 (primary), 20E06, 20E07, 20E08, 20M05, 46L35 (secondary)}

\begin{document}

\begin{abstract}
We investigate criteria ensuring that a one-relator group $G$ contains a right-angled Artin subgroup $A(\Gamma)$, corresponding to a finite graph $\Gamma$. In particular, we prove that if $\Gamma$ is a forest with at least one edge and the positive submonoid $T(\Gamma)$, of $A(\Gamma)$, embeds into $G$ then so does all of $A(\Gamma)$. As by-products of our methods we obtain characterisations of one-relator groups that have property $P_{nai}$ and that are $C^*$-simple.
\end{abstract}

\maketitle
\section{Introduction}
Given a simplicial graph $\Gamma$ with vertex set $V(\Gamma)$ and edge set $E(\Gamma)$, one can look at the following presentation:
\begin{equation}\label{eq:pres_of_a_RAAG} 
    \langle V(\Gamma) \mid uv=vu,\text{ provided } \{u,v\} \in E(\Gamma) \rangle.
\end{equation}
Considered as a presentation of a group, \eqref{eq:pres_of_a_RAAG} defines the \emph{right-angled Artin group} $A(\Gamma)$, corresponding to the graph $\Gamma$. But presentation \eqref{eq:pres_of_a_RAAG} can also be considered as a monoid presentation, in which case it defines the \emph{trace monoid} $T(\Gamma)$. Right-angled Artin groups (a.k.a., \emph{graph groups} or \emph{partially commutative groups}) occupy the central stage in Geometric Group Theory because of their rich structure of subgroups, nice geometric properties and their key  role in the theory of special cube complexes \cite{Charney,Wise-book}. On the other hand, trace monoids have been used in Computer Science as an important algebraic model for studying paralellism \cite{Diekert}. It is known \cite{Paris} that the trace monoid $T(\Gamma)$ is naturally isomorphic to the submonoid of $A(\Gamma)$ generated by $V(\Gamma)$, i.e., $T(\Gamma) \cong A(\Gamma)^+$ is the positive submonoid of $A(\Gamma)$.

A classical theme of infinite Group Theory is the study of 
\emph{one-relator groups}, that is, groups defined by a presentation 
of the form 
\begin{equation}\label{eq:1-rel_pres} 
G=\langle x_1,\dots , x_k \mid W=1 \rangle,    
\end{equation}
where  $k \in \N\cup \{0\}$ and $W$ is a cyclically reduced word in the free group $F$, freely generated by $\{x_1,\ldots,x_k\}$. Embeddings of trace monoids and right-angled Artin groups into one-relator groups have been investigated by Gray \cite{Gray} and Foniqi, Gray and Nyberg-Brodda \cite{FGNB}, who were motivated by decision problems in one-relator groups and one-relation monoids. In particular, such embeddings were used in \cite{Gray} to construct one-relator groups with undecidable rational subset membership problem and one-relator inverse monoids with unsolvable word problem. This is further developed in \cite{FGNB}, where it is shown that the submonoid membership problem and the prefix membership problem need not be decidable even in some ``nice'' specific subclasses of one-relator groups.

Let $P_n$ denote the path with $n$ vertices, of length $n-1$. For example, when $n=4$, $A(P_4)$ can be described by the following presentation:
\begin{equation}\label{eq:pres-P_4}
    A(P_4)=\langle \alpha,\beta,\gamma,\delta \mid \alpha \beta=\beta\alpha,~\beta \gamma=\gamma\beta,~\gamma\delta=\delta\gamma \rangle.
\end{equation} 
This presentation can also be considered as a monoid presentation, defining the trace monoid $T(P_4)$. The monoid $T(P_4)$ and the group $A(P_4)$ seem to occupy a special place in the discussion of submonoid and rational subset membership problem, see  \cite{LohSte} and \cite{FGNB}.
Motivated by this, in \cite[Question~6.13]{FGNB} the authors ask whether there exists a one-relator group that contains a submonoid isomorphic to $T(P_4)$ but no subgroups isomorphic to $A(P_4)$. We answer a more general version of this question in the negative, as follows.

\begin{thm}\label{thm:main} Let $G$ be a one-relator group and let $\Gamma$ be a finite forest with at least one edge. If $T(\Gamma)$ embeds into $G$ then so does $A(\Gamma)$.
\end{thm}

\begin{rem}\label{rem:RAAG_emb_in_1-rel_iff_forest} It is known that for a finite graph $\Gamma$, $A(\Gamma)$ embeds in a one-relator group if and only if $\Gamma$ is a forest (see \cite[Theorem~2.2 and Remark~2.3]{Gray}).  
\end{rem}

\begin{rem} The condition that $\Gamma$ has at least one edge in Theorem~\ref{thm:main} is equivalent to the statement that $T(\Gamma)$ is not a free monoid. While this condition is essential (see Example~\ref{ex:BS} below), it can be replaced by the assumption that $G$ is not a solvable Baumslag--Solitar group, because all other non-cyclic one-relator groups contain non-abelian free subgroups.    
\end{rem}

\begin{ex}\label{ex:BS} Theorem~\ref{thm:main} does not generalise to the case when $\Gamma$ is totally disconnected (i.e., when $E(\Gamma)=\varnothing$). For example, the Baumslag--Solitar group $G = BS(1,2) = \langle a,t \mid tat^{-1} = a^2 \rangle$ is solvable and, therefore, it does not contain non-abelian free subgroups. However,  $G$ does contain a free submonoid of rank $2$: this follows from the work of Rosenblatt \cite[Theorem~4.7]{Rosen} since $G$ is not polycyclic. It is also not hard to verify directly that $t$ and $at$ freely generate a free submonoid in $G$.
\end{ex}

It is easy to see that for a finite tree $\Gamma$ of diameter at most $2$, $A(\Gamma)$ embeds into $A(P_3)$ (see Lemma~\ref{lem:tree_of_small_diam}). On the other hand, in \cite{KimKoberda} Kim and Koberda proved that $A(P_4)$ contains subgroups isomorphic to $A(\Gamma)$ for any finite forest $\Gamma$. Therefore, in the case when $\Gamma$ is connected Theorem~\ref{thm:main} can be quickly obtained from the following result.

\begin{thm}\label{thm:P_n}
Let $n = 3$ or $n = 4$. If a one-relator group $G$ contains $T(P_n)$ then it also contains $A(P_n)$. Moreover, if elements $a_1,\dots,a_n$ generate a submonoid isomorphic to $T(P_n)$ in $G$ then the subgroup $\langle a_1,\dots,a_n\rangle \leqslant G$ contains a copy of $A(P_n)$.
\end{thm}

The case $n=3$ in Theorem~\ref{thm:P_n} is much easier and can be deduced from a theorem of Bieri \cite[Corollary~8.7]{Bieri} stating that the quotient of a finitely generated group of cohomological dimension $2$ by a non-trivial normal subgroup is virtually free, although we give an independent argument using actions on trees. However, we are not aware of any general results that would also cover the case $n=4$. In fact, we prove the following much stronger statement.

\begin{prop} \label{prop:P_4-enhanced}
Let $G$ be a one-relator group and let $b,c \in G$ be elements generating a copy of $\Z^2$ in $G$. Suppose that there are elements $a \in \mathrm{C}_G(b)$ and $d \in \mathrm{C}_G(c)$ such that the following three conditions are satisfied:
\begin{enumerate}[label=\rm(\roman*)]
    \item $a\langle b,c \rangle a^{-1} \cap \langle b,c \rangle=\langle b \rangle$;
    \item $\langle b,c \rangle \cap d^{-1} \langle b,c \rangle d=\langle c \rangle$;
    \item $a\langle b,c \rangle a^{-1} \cap d^{-1} \langle b,c \rangle d=\{1\}$.
\end{enumerate}
Then there exist integers $k,l,m,n \in \Z\setminus\{0\}$ such that the elements $ac^ka^{-1}$, $b^l$, $c^m$, $d^{-1}b^n d$ naturally generate a copy of $A(P_4)$ in $G$. More precisely, the assignment $\alpha \mapsto ac^ka^{-1}$, $\beta \mapsto b^l$, $\gamma \mapsto c^m$ and $\delta \mapsto d^{-1}b^n d$ extends to an injective homomorphism $A(P_4) \to G$, where $A(P_4)$ is given by \eqref{eq:pres-P_4}.    
\end{prop}
This proposition says that any homomorphism from $A(P_4)$ to a one-relator group $G$ that is injective on the union of  three abelian subgroups $\alpha \langle \beta, \gamma \rangle \alpha^{-1} \cup \langle \beta, \gamma \rangle \cup \delta^{-1} \langle \beta, \gamma \rangle \delta $    can be ``promoted'' to an injective embedding $A(P_4)\hookrightarrow G$ (see Proposition~\ref{prop:n=4}).

The proofs of Proposition~\ref{prop:P_4-enhanced} and Theorem~\ref{thm:P_n} rely on the action of a one-relator group $G$ on its Bass--Serre tree corresponding to a splitting of $G$ as an HNN-extension of some ``simpler'' one-relator group, as described in Subsection~\ref{ssec:prelim-1rel}. We analyse different cases that can occur, using the classification of isometries of trees, connectedness of $P_3$ and $P_4$ and various ping-pong lemmas.
However, if $\Gamma$ is disconnected and the diameters of its connected components are all small (less than $3$), these arguments do not work, as demonstrated by Example~\ref{ex:BS}. To get around this issue we study free products in one-relator groups.

Recall that a non-trivial group $G$ is said to be \emph{a generalised Baumslag--Solitar group} if it splits as the fundamental group of a finite graph of groups where all vertex and edge groups are infinite cyclic.

\begin{thm} \label{thm:one-relator-free-or-snormal}
Suppose that $G$ is a one-relator group that is not cyclic and not a generalised Baumslag--Solitar group. If $A_1,\dots,A_s \leqslant G$ are subgroups with non-trivial centres then there exists an infinite order  element $f \in G$ such that the subgroups $\langle f,A_i \rangle \leqslant G$ are naturally isomorphic to the free product $\langle f \rangle * A_i$, for all $i=1,\dots,s$.
\end{thm}

\begin{rem}\label{rem:at_least_3_gens}
If $G$ is a group given by presentation \eqref{eq:1-rel_pres} with $k \ge 3$ then it is acylindrically hyperbolic by \cite[Corollary~2.6]{MinasyanOsin}, in particular it is neither cyclic nor a generalised Baumslag--Solitar group by \cite[Corollary~1.5]{Osin}, because any generalised Baumslag--Solitar group commensurates an infinite cyclic subgroup. Hence $G$ satisfies the assumptions of Theorem~\ref{thm:one-relator-free-or-snormal}.    
\end{rem}

Theorem~\ref{thm:one-relator-free-or-snormal} allows us to prove Theorem~\ref{thm:main} in the case when $\Gamma$ is disconnected, using the observation that $A(\Gamma)$ embeds into $A(P_{d+1})*\Z$, where $d \in \N$ is the maximum of the diameters of the connected components of $\Gamma$ (see Lemma~\ref{lem:free-products-embed}). 

Theorem~\ref{thm:one-relator-free-or-snormal} is interesting on its own right.
In \cite{Bek-Cow-Har} Bekka, Cowling and de la Harpe introduced the following property of a group $G$: $G$  has \emph{property $P_{nai}$} if for any finite subset $\{a_1,\dots,a_s\} \subseteq G\setminus\{1\}$ there is an infinite order element $f \in G$ such that the subgroup $\langle f,a_i \rangle$ is isomorphic to the free product $\langle f \rangle*\langle a_i \rangle$, for each $i=1,\dots, s$. This property is known for many groups, in particular, for all acylindrically hyperbolic groups without non-trivial finite normal subgroups \cite{AbbottDahmani}, and hence, for all one-relator groups with at least $3$ generators. 
Theorem~\ref{thm:one-relator-free-or-snormal} allows us to treat the case of $2$-generated one-relator groups.

\begin{cor}\label{cor:P_nai}
Let $G$ be a non-cyclic one-relator group. Then $G$ satisfies property $P_{nai}$ if and only if $G$ is not isomorphic to a generalised Baumslag--Solitar group. 
\end{cor}

Any generalised Baumslag--Solitar group commensurates an infinite cyclic subgroup, whence it does not have $P_{nai}$.  Thus Corollary~\ref{cor:P_nai} completes the description of one-relator groups with $P_{nai}$.

Historically, $P_{nai}$ was introduced as a condition that implies the simplicity of the reduced $C^*$-algebra of the group \cite{Bek-Cow-Har}. And although much more powerful criteria for $C^*$-simplicity have recently been discovered  by Kalantar and Kennedy \cite{Kal-Ken} and Breuillard, Kalantar, Kennedy and Ozawa \cite{BKKO}, the following result seems to be new.

\begin{cor}\label{cor:C*-simple}
Let $G$ be a non-cyclic one-relator group given by \eqref{eq:1-rel_pres}. Then $G$ is not $C^*$-simple if and only if $k=2$ and at least one of the following holds:
\begin{enumerate}[label=\rm(\roman*)]
    \item \label{it:C*-BS1n} $G$ is isomorphic to a solvable Baumslag--Solitar group $BS(1,n)=\langle a,t \mid tat^{-1}=a^n \rangle$, for some $n \in \Z\setminus\{0\}$;
    \item \label{it:C*-unimod} $G$ is a unimodular generalised Baumslag--Solitar group. In this case $G$ contains an infinite cyclic normal subgroup and a finite index subgroup $K \cong F \times \Z$, for some finitely generated free group $F$.
\end{enumerate}
\end{cor}

Corollary~\ref{cor:C*-simple} gives a complete characterisation of $C^*$-simplicity for  one-relator groups. 
In view of Corollary~\ref{cor:P_nai}, the proof of Corollary~\ref{cor:C*-simple} amounts to investigating which generalised Baumslag--Solitar groups are $C^*$-simple. We do this in Proposition~\ref{prop:GBS-C*-simple}, using work of de la Harpe and Pr\'{e}aux \cite{dlH-Pre}, as well as Brownlowe, Mundey, Pask, Spielberg and Thomas \cite{BMPST}.

Corollary~\ref{cor:C*-simple} can be made effective, i.e., given a presentation \eqref{eq:1-rel_pres} of a group $G$ it is possible to decide whether or not $G$ is $C^*$-simple: see Remark~\ref{rem:deciding_if_C*-simpl}.

In Section~\ref{sec:prelims}, we review the necessary preliminaries on group actions on trees, one-relator groups, trace monoids and right-angled Artin groups. We prove the cases $n=3$ and $n=4$ of Theorem~\ref{thm:P_n} in Sections \ref{sec:P_3}~and~\ref{sec:P_4}, respectively. In Section~\ref{sec:asc_HNN} we establish a version of Theorem~\ref{thm:one-relator-free-or-snormal} for strictly ascending HNN-extensions of free groups, which is the main new ingredient in the proof of this theorem in Section~\ref{sec:one-rel}. We prove Theorem~\ref{thm:main}  in Section~\ref{sec:ThmMain}, and in Section~\ref{sec:F2F2F2} we give various examples showing that generators of $T(P_n)$ need not generate a copy of $A(P_n)$ in a one-relator group; we also prove that $T(P_4)$ embeds into the direct cube of the free monoid of rank $2$ (Lemma~\ref{lem:F2F2F2}), which allows us to exhibit examples of groups containing $T(P_4)$ but not containing $A(P_4)$. Finally, in Section~\ref{sec:C*-simple}
we study the $C^*$-simplicity of generalised Baumslag--Solitar groups and prove Corollary~\ref{cor:C*-simple}.

\subsection*{Acknowledgements}
The authors would like to thank Robert Gray and Alexander Zakharov for valuable discussions, and the anonymous referee for useful comments. The second author was partially supported by the National Science Centre (Poland) grant 2022/47/D/ST1/00779.

\section{Preliminaries}\label{sec:prelims}

\subsection{Group actions on trees}\label{subsec:trees}

One of the main tools in this paper is the theory of group actions on (simplicial) trees, see \cite[Section~I.4]{DD} for the basics of this theory. Let $G$ be a group acting by isometries on a tree $\cT$. We will always assume that this action does not invert any edges of $\cT$. Then every isometry of $\cT$ is either \emph{hyperbolic} or \emph{elliptic} (see, for example, \cite[Proposition~I.4.11]{DD}). A hyperbolic element $g \in G$ has infinite order and possesses a unique \emph{axis} $\ax(g) \subseteq \cT$, which is an embedded simplicial line on which $g$ acts by translation. The \emph{translation length} $\|g\| \in \mathbb{N}$, of $g$, is the distance by which every vertex of $\ax(g)$ is shifted by $g$ (here $\cT$ is equipped with the standard edge-path metric).
Any element commuting with $g$ must preserve $\ax(g)$ setwise.

An elliptic element $g \in G$ fixes at least one vertex of $\cT$, and we will use $\fix(g)$ to denote the unique non-empty maximal subtree of $\cT$ fixed by $g$ pointwise. Again, any element of $G$ commuting with $g$ will leave $\fix(g)$ invariant.
The relationships between axes and fixed points of commuting elements are summarised in the following statement.

\begin{lemma} \label{lem:commuting-types} Suppose that $G$ is a group acting on a tree $\cT$, and $g,h \in G$ are two commuting elements. Then the following hold.
\begin{enumerate}[label=\rm(\roman*)]
\item \label{it:commtypes-preserved} The element $g$ preserves $\operatorname{Min}(h)$ setwise, where $\operatorname{Min}(h) = \fix(h)$ if $h$ is elliptic and $\operatorname{Min}(h) = \ax(h)$ if $h$ is hyperbolic.
\item \label{it:commtypes-hyperbolic} If $g$ and $h$ are hyperbolic then $\ax(g) = \ax(h)$.
\item \label{it:commtypes-hypell} If $g$ is hyperbolic and $h$ is elliptic then $\ax(g) \subseteq \fix(h)$.
\item \label{it:commtypes-elliptic} If $g$ and $h$ are elliptic then $\fix(g) \cap \fix(h) \neq \varnothing$. Furthermore, if the $G$-stabilisers of edges in $\cT$ are free and $\langle g,h \rangle$ is not cyclic, then $\fix(g) \cap \fix(h)$ is a single vertex of $\cT$.

\end{enumerate}
\end{lemma}
The second part of claim \ref{it:commtypes-elliptic} follows from the fact that free groups do not contain non-cyclic abelian subgroups.

Recall that an \emph{end} of a tree $\cT$ is an equivalence class of (infinite geodesic) rays starting at  vertices of $\cT$, where two such rays are equivalent if their intersection is another ray. The \emph{boundary} $\partial\cT$ is defined as the set of all ends of $\cT$. Clearly the action of $G$ on $\cT$ naturally induces a $G$-action on $\partial \cT$. Given any $\varepsilon \in \partial \cT$, any ray from this equivalence class is said to \emph{converge to $\varepsilon$}. Observe that for any vertex $v$ of $\cT$ and any end $\varepsilon \in \partial \cT$ there exists a unique ray starting at $v$ and converging to $\varepsilon$. 

The next lemma is an elaboration of the general statement about a group acting on a tree with a fixed end (see, for example, \cite[Lemma~4.9]{MinasyanOsin}), adapted to the context of this paper.

\begin{lemma}\label{lem:fixed_end} Let $G$ be a group acting on a tree $\cT$ with free edge stabilisers, and let
$A$ be a subgroup of $G$ fixing some end of $\cT$. Suppose that there exists an elliptic element $z \in A \setminus\{1\}$ such that $z$ is central in $A$. Then
\begin{enumerate}[label=\rm(\roman*)]
    \item \label{it:fixedend-normal} the elliptic elements of $A$ form a normal subgroup $N \lhd A$ and this subgroup is torsion-free and locally cyclic;
    \item \label{it:fixedend-Z^2} if $A$ contains a hyperbolic element then $N \cong \Z$ and $A \cong \Z^2$.
\end{enumerate}
\end{lemma}

\begin{proof} Let $\varepsilon \in \partial \cT$ be the end fixed by $A$. Let $N$ denote the set of all elliptic elements in $A$. Observe that any elliptic element $h \in A$ pointwise fixes the ray starting at an arbitrary vertex $v \in \fix(h)$ and converging to $\varepsilon$. Therefore any two elliptic elements $g,h \in A$ stabilise pointwise some ray converging to $\varepsilon$, so the product $gh$ is also elliptic. Since inverses and conjugates of elliptic elements are elliptic we can conclude that  $N$ is a normal subgroup of $A$.

Choose any ray $\mathcal R$  converging to $\varepsilon$ in $\cT$ and let $e_1,e_2,\dots$ be the edges of this ray in the order of their appearance from the origin of the ray. Then the pointwise stabiliser $\st_A(e_i)$ is contained in $\st_A(e_j)$, for any $j \ge i$, because $A$ fixes the end $\varepsilon$. Any elliptic element of $A$ fixes all but finitely many edges of $\mathcal R$, which implies that $N$ is the ascending union
\begin{equation}\label{eq:struct_of_N}
N=\bigcup_{i=1}^\infty \st_A(e_i).    
\end{equation}
For all sufficiently large $i \in \N$, $\st_A(e_i)$ is free and contains the non-trivial central element $z$, hence this stabiliser must be infinite cyclic, and so \eqref{eq:struct_of_N} shows that $N$ is torsion-free and locally cyclic, proving part~\ref{it:fixedend-normal}.

Now, suppose that $A$ contains at least one hyperbolic element. Let $h \in A$ be a hyperbolic element with minimal translation length. For any other hyperbolic element $g \in A$ the intersection $\ax(g) \cap \ax(h)$ must contain a ray $\mathcal S$ converging to $\varepsilon$ in $\cT$ (because hyperbolic elements only fix the ends of their axes in $\partial \cT$). If the translation length $\|h\|$ does not divide the translation length $\|g\|$ then we can find $m, n \in \Z\setminus\{0\}$ such that $m\|h\|+n\|g\|=\mathrm{gcd}(\|h\|,\|g\|)<\|h\|$, so the element $h^mg^n \in A$ is hyperbolic with translation length strictly smaller than $\|h\|$, contradicting the choice of $h$. Therefore $\|h\|$ must divide $\|g\|$, so there exists $m \in \Z$ such that the element $h^m g$ fixes all but finitely many edges of the ray $\mathcal S$, and hence it belongs to~$N$. Thus we have shown that $g \in \langle h \rangle N$, and therefore $A=\langle h \rangle N$.

Let $\mathcal R$ be any ray contained in $\ax(h)$ and converging to $\varepsilon$. After replacing $h$ with its inverse, if necessary, we can assume that $h$ translates  $\mathcal{R}$ into itself. Let $e_1$ be the first edge of $\mathcal R$. The argument above \eqref{eq:struct_of_N} shows that 
\begin{equation}  \label{eq:nested}
 \st_A(e_1) \subseteq \st_A(h\, e_1)= h\st_A(e_1)h^{-1}, \text{ and }   
\end{equation}
\begin{equation} \label{eq:refined_struc_of_N}
N=\bigcup_{j=0}^\infty \st_A(h^j\, e_1)=\bigcup_{j=0}^\infty h^j \st_A(e_1) h^{-j}.    
\end{equation}

Since  $z \in \st_A(\ax(h))\subseteq \st_A(e_1)$ by Lemma~\ref{lem:commuting-types}, $\st_A(e_1)$ is a free group with non-trivial centre, so it must be infinite cyclic, generated by some $w \in A$. Inclusion \eqref{eq:nested} shows that $w=hw^nh^{-1}$, for some $n \in \Z\setminus\{0\}$. Since $h$ commutes with $z \in \langle w \rangle \setminus\{1\}$, we can deduce that $n=1$, i.e., $h$ centralizes $\langle w \rangle=\st_A(e_1)$. Equation \eqref{eq:refined_struc_of_N} now yields that $N=\langle w \rangle \cong \Z$, whence $A=\langle h \rangle \langle w \rangle \cong \Z^2$, proving part~\ref{it:fixedend-Z^2}.
\end{proof}

\begin{lemma}\label{lem:bounded_intersec_of_axes} Consider a group $G$ acting on a tree $\cT$ with free edge stabilisers. Suppose that $a,c \in G$ are hyperbolic elements and $b \in G$ is an elliptic element commuting with $a$ and $c$. If $\langle a,b \rangle \cong \Z^2$ and $\langle a,b \rangle \cap \langle c \rangle=\{1\}$ in $G$ then $\ax(a) \cap \ax(c)$ is bounded in $\cT$.   
\end{lemma}

\begin{proof} Arguing by contradiction, assume that the intersection $\ax(a) \cap \ax(c)$ is unbounded. Then this intersection contains a ray converging to an end of $\cT$, and this end will be fixed by the subgroup $A=\langle a,b,c \rangle$ (see Lemma~\ref{lem:commuting-types}.\ref{it:commtypes-hypell}). By the assumptions, $A$ contains a central elliptic element $b \neq 1$ and hyperbolic elements $a$ and $c$, so we can apply Lemma~\ref{lem:fixed_end} to conclude that $A \cong \Z^2$. Since $\langle a,b \rangle \cong \Z^2$,  this subgroup must have finite index in $A$. Therefore, there is $n \in \N$ such that $c^n \in \langle a,b \rangle$, contradicting the assumption that $\langle a,b \rangle \cap \langle c \rangle=\{1\}$. Thus the lemma is proved.    
\end{proof}

\subsection{One-relator groups}
\label{ssec:prelim-1rel} Let $G$ be a finitely generated one-relator group given by \eqref{eq:1-rel_pres}.
We use the so-called Magnus--Moldavanskii hierarchy for $G$ in the strong form, recently established by Masters \cite{Masters} and Linton \cite{Linton}: there exists a sequence of finitely generated one-relator groups
\begin{equation} \label{eq:HNN}
G_0\hookrightarrow G_1 \hookrightarrow  G_2\hookrightarrow  \dots \hookrightarrow G_s=G,
\end{equation}
where $G_0$ is a free product of finitely many cyclic groups and $G_{i}$ is an HNN-extension of $G_{i-1}$ with finitely generated free associated subgroups, for every $i=1,\dots,s$.

\begin{rem}
We only employ the strong form of Magnus--Moldavanskii hierarchy for convenience, and for our purposes the classical form (see \cite[Section~IV.5]{LyndonSchupp}) would also be sufficient. In this classical hierarchy the group $G_{i}$ is only known to embed into an HNN-extension of a ``simpler'' one-relator group $G_{i-1}$.     
\end{rem}

For the remainder of this subsection we fix a one-relator group $G$ with presentation \eqref{eq:1-rel_pres} and assume that it has a hierarchy \eqref{eq:HNN} with $s>0$. We let $\mathcal T$ be the Bass--Serre tree for the splitting of $G=G_s$ as an HNN-extension of $G_{s-1}$. Then $G$ acts on $\cT$ by isometries and without edge inversions, with one orbit of vertices and two orbits of (oriented) edges. The vertex stabilisers are conjugate to $G_{s-1}$ in $G$ and edge stabilisers are free of finite rank. In particular, the statements established in Subsection~\ref{subsec:trees} all apply in this situation.

Furthermore, it follows from the construction of the sequence~\eqref{eq:HNN} that if $H$ is the stabiliser of an edge of $\cT$ incident to a vertex stabilised by $G_{s-1}$, then $H$ is conjugate to a \emph{Magnus subgroup} of $G_{s-1}$, i.e., a free subgroup generated by a subset of $\{x_1,\dots,x_k\}$ that omits at least one letter $x_i$ involved in the relator $W$.  We thus have the following geometric rephrasing of a result of Bagherzadeh.

\begin{prop}[{\cite[Theorem B]{Bagh}}] \label{prop:bagh}
Let $e$ be an edge starting at a vertex $v$ in $\cT$. Then for any $g \in \st_G(v) \setminus \st_G(e)$, the intersection $\st_G(e) \cap \st_G(g\,e)$ is cyclic.
\end{prop}

\subsection{Ping-pong lemmas}
\label{ssec:ping-pong}
Many of our arguments rely on variations of the classical Ping-Pong Lemma to exhibit free products (possibly with amalgamation) inside groups acting on trees. We will use two versions of the lemma: one for a free product of finitely many groups and one for a free product of two groups with amalgamation.

\begin{prop}[Ping-Pong Lemma, version 1 {\cite[Lemma~2.1]{O-S}}] \label{prop:PPv1}
Let $G$ be a group generated by its subgroups $G_1,\ldots,G_k$, for some $k \geq 2$, with $|G_i| \geq 3$, for some $i$. Suppose that $G$ acts on a set $\Omega$ and that there exist pairwise disjoint non-empty subsets $\Omega_1,\ldots,\Omega_k \subset \Omega$ satisfying the following property:
\begin{itemize}
\item $g_i\,\Omega_j \subseteq \Omega_i$, for all $i \neq j$ and all $g_i \in G_i\setminus\{1\}$.
\end{itemize}
Then $G$ naturally splits as the free product $G_1 * \dots * G_k$.
\end{prop}

\begin{prop}[Ping-Pong Lemma, version 2 {\cite[Proposition~III.12.4]{LyndonSchupp}}] \label{prop:PPv2}
Let $G$ be a group generated by its subgroups $G_1$ and $G_2$, so that $H = G_1 \cap G_2$ is a proper subgroup of both $G_1$ and $G_2$, and $[G_i:H] \geq 3$ for some $i \in \{1,2\}$. Suppose that $G$ acts on a set $\Omega$ and that there exist pairwise disjoint non-empty subsets $\Omega_1,\Omega_2 \subset \Omega$ satisfying the following properties:
\begin{itemize}
\item $g_i\,\Omega_j \subseteq \Omega_i$, for $i \neq j$ and all $g_i \in G_i \setminus H$;
\item $h\,\Omega_i \subseteq \Omega_i$, for each $i = 1,2$ and all $h \in H$.
\end{itemize}
Then $G$ naturally splits as the amalgamated free product $G_1 *_H G_2$.
\end{prop}

If $H$ a free group on a subset $X \subseteq H$ then the Cayley graph of $H$ with respect to $X$ is a tree $\cT$ and $H$ acts freely on $\cT$. Every non-trivial element $h \in H$ will act as a hyperbolic isometry on $\cT$, fixing two distinct points $h^\infty$ and $h^{-\infty}$ on the boundary $\partial \cT$. It is easy to see that if $g,h \in H \setminus\{1\}$ then $\{g^{\pm \infty}\} \cap \{h^{\pm \infty}\} \neq \emptyset$ if and only if $g^m=h^n$ for some $m,n \in \Z\setminus\{0\}$. 
A standard application of the ping-pong argument from Proposition~\ref{prop:PPv1}, applied to the action of $H$ on $\partial \cT$, gives the following (see \cite[Corollary~6]{Olsh} or the proof of \cite[Proposition~III.$\Gamma$.3.20]{BridsonHaefliger} for a more general statement in the case of hyperbolic groups).

\begin{lemma}\label{lem:large powers} Let $H$ be a free group and let $h_1,\dots,h_k \in H$ be  non-trivial elements. If $\langle h_i \rangle \cap \langle h_j \rangle=\{1\}$ for all $i<j$, then there exist $r \in \N$ such that the elements $h_1^r,\dots,h_k^r$ freely generate a free subgroup of rank $k$ in $H$.    
\end{lemma}

\subsection{Trace monoids and right-angled Artin groups}
\label{ssec:prelim-raags}

Let $\Gamma$ be a finite simplicial graph, and denote by $V(\Gamma)$ and $E(\Gamma)$ the sets of vertices and edges of $\Gamma$, respectively. The \emph{right-angled Artin group} on $\Gamma$, denoted by $A(\Gamma)$, is the group given by the presentation
\begin{equation} \label{eq:pres_RAAG}
\langle v \in V(\Gamma) \mid vw = wv, \text{ for all } \{v,w\} \in E(\Gamma) \rangle.
\end{equation}
Moreover, the \emph{trace monoid} on $\Gamma$, denoted by $T(\Gamma)$, is the monoid given by the presentation \eqref{eq:pres_RAAG}, viewed now as a monoid presentation. It is a result of Paris \cite[Theorem~1.1]{Paris} that $T(\Gamma)$ is naturally isomorphic to the \emph{monoid of positive words} $A(\Gamma)^+$ in $A(\Gamma)$, whose elements can be represented by positive words over $V(\Gamma)$ (that do not involve letters from $V(\Gamma)^{-1}$).

If $\Delta$ is an induced subgraph of $\Gamma$ then $V(\Delta)$ naturally generates a copy of $A(\Delta)$ in $A(\Gamma)$, and a copy of $T(\Delta)$ in $T(\Gamma)$. This important property can be proved by noticing that $A(\Gamma)$ (or $T(\Gamma)$) retracts onto $A(\Delta)$ (respectively, $T(\Delta)$), via the map sending all generators from $V(\Gamma) \setminus V(\Delta)$ to the identity element.

In most of the arguments, we will focus on the case when $\Gamma$ is a path $P_n$, of length $n-1$, for $n \in \{3,4\}$. Then $A(P_n)$ (respectively $T(P_n)$) has the group presentation (respectively monoid presentation)
\begin{equation}\label{eq:pres_HN}
\langle \alpha_1,\ldots,\alpha_n \mid \alpha_1\alpha_2=\alpha_2\alpha_1,\ldots,\alpha_{n-1}\alpha_n=\alpha_n\alpha_{n-1} \rangle.    
\end{equation}
By saying that elements $a_1,\ldots,a_n$ in a group $G$ generate a copy of $A(P_n)$ (respectively, $T(P_n)$) \emph{in the natural way}, we mean that the assignment $\alpha_i \mapsto a_i$, with $\alpha_1,\ldots,\alpha_n$ as in \eqref{eq:pres_HN}, extends to an injective group homomorphism $A(P_n) \to G$ (respectively, an injective monoid homomorphism $T(P_n) \to G$).

Given a finite simplicial graph $\Gamma$, the \emph{Bestvina--Brady subgroup} of $A(\Gamma)$ is the kernel of the group homomorphism $\Phi\colon A(\Gamma) \to \Z$ defined by setting $\Phi(v) = 1$, for all $v \in V(\Gamma)$. Since $P_n$ is a tree for all $n \in \N$, we have the following special case of a result of Dicks and Leary \cite{DicksLeary}.

\begin{lemma}[{\cite[Theorem~1]{DicksLeary}}] \label{lem:BBfree}
Let $n \in \N$ and let $H \lhd A(P_n)$ be the Bestvina--Brady subgroup. Then $H$ is free of rank $n-1$, freely generated by the elements $\alpha_1^{-1}\alpha_2,\ldots,\alpha_{n-1}^{-1}\alpha_n \in H$.
\end{lemma}

\section{Generating \texorpdfstring{$A(P_3)$}{A(P\_3)}}
\label{sec:P_3}

In this section we prove a strong version of Theorem~\ref{thm:P_n} in the case $n = 3$. Recall that $A(P_3)$ has the following standard presentation:
\begin{equation}\label{eq:pres-P_3}
    A(P_3)=\langle \alpha,\beta,\gamma \mid \alpha \beta=\beta\alpha,~\beta \gamma=\gamma\beta \rangle.
\end{equation} In particular, $A(P_3) \cong F_2 \times \Z$, where the free group $F_2$ is freely generated by $\{\alpha,\gamma\}$, and the infinite cyclic factor is generated by $\beta$. It is easy to see, for example by looking at the abelianisation of $A(P_3)$, that the subgroups  $\langle \alpha,\beta \rangle, \langle \beta,\gamma\rangle \leqslant A(P_3)$ are free abelian of rank $2$, and 
\[ \langle \alpha,\beta \rangle \cap \langle \beta , \gamma\rangle= \langle \beta \rangle \text{ in } A(P_3).\]

\begin{prop} \label{prop:P3} Suppose that $G$ is a one-relator group and we have a homomorphism $\psi\colon A(P_3) \to G$ (where $A(P_3)$ is given by the presentation \eqref{eq:pres-P_3}) such that $\psi$ is injective on the union  $\langle \alpha,\beta \rangle \cup \langle \beta,\gamma\rangle$. Then there exist 
 $k,l,m \in \mathbb{Z} \setminus \{0\}$ such that the elements $\psi(\alpha)^k$, $\psi(\beta)^l$ and $\psi(\gamma)^m$ generate a copy of $A(P_3)$ in $G$ in the natural way.
\end{prop}

\begin{proof}
Let $a,b,c \in G$ denote the $\psi$-images of the generators $\alpha, \beta, \gamma$ of $A(P_3)$ respectively. Note that each of these elements has infinite order in $G$.
Consider the hierarchy \eqref{eq:HNN} for $G=G_s$. Notice that $s>0$ because $G_0$ is a free product of finitely many cyclic groups, so it cannot contain free abelian subgroups of rank $2$. Thus we may assume, by induction on $s$, that the subgroup $\langle a,b,c \rangle$ is not conjugate into $G_{s-1}$ in $G$. In particular, this subgroup does not fix a vertex of the Bass--Serre tree $\cT$, associated to the HNN-splitting of $G$ over $G_{s-1}$, as discussed in Subsection~\ref{ssec:prelim-1rel}.

We aim to prove that for some $k,l,m \in \mathbb{Z}\setminus\{0\}$ the group homomorphism $\varphi = \varphi_{k,l,m}\colon A(P_3) \to G$, sending $\alpha \to a^k$, $\beta \to b^l$ and $\gamma \to c^m$, is injective. 
It is enough to show that $a^kb^{pl}$ and $b^{ql}c^m$ freely generate the free subgroup $F_2$, of rank $2$, in $G$, for some $p,q \in \mathbb{Z}$: that is, $\varphi|_{\langle \alpha\beta^p,\beta^q\gamma \rangle}$ is injective. Indeed, suppose this is the case and let $g \in \ker(\varphi)$. Let $r \in \mathbb{Z}$ and $h \in \langle \alpha\beta^p,\beta^q\gamma \rangle$ be such that $g = \beta^r h$. Then $\varphi(h) = \varphi(\beta)^{-r}$ is central in $\varphi(\langle \alpha\beta^p,\beta^q\gamma \rangle) \cong F_2$, implying that $\varphi(h) = 1$, and, therefore, $h = 1$ since $\varphi|_{\langle \alpha\beta^p,\beta^q\gamma \rangle}$ is injective. Thus $g = \beta^r$, implying that $g = 1$ (as if $r \neq 0$, then $\varphi(\beta) = \psi(\beta)^l$ has finite order in $G$, contradicting the assumption that $\psi$ is injective on $\langle \alpha,\beta \rangle \cong \Z^2$). Thus $\varphi$ must be injective, as required.

The proof depends on the types of isometries that $a$, $b$ and $c$ induce on $\cT$. In each case, we will use Lemma~\ref{lem:commuting-types} to analyse the resulting configuration of axes and fixed-point sets. Some of the cases considered are displayed in Figure~\ref{fig:P3}.

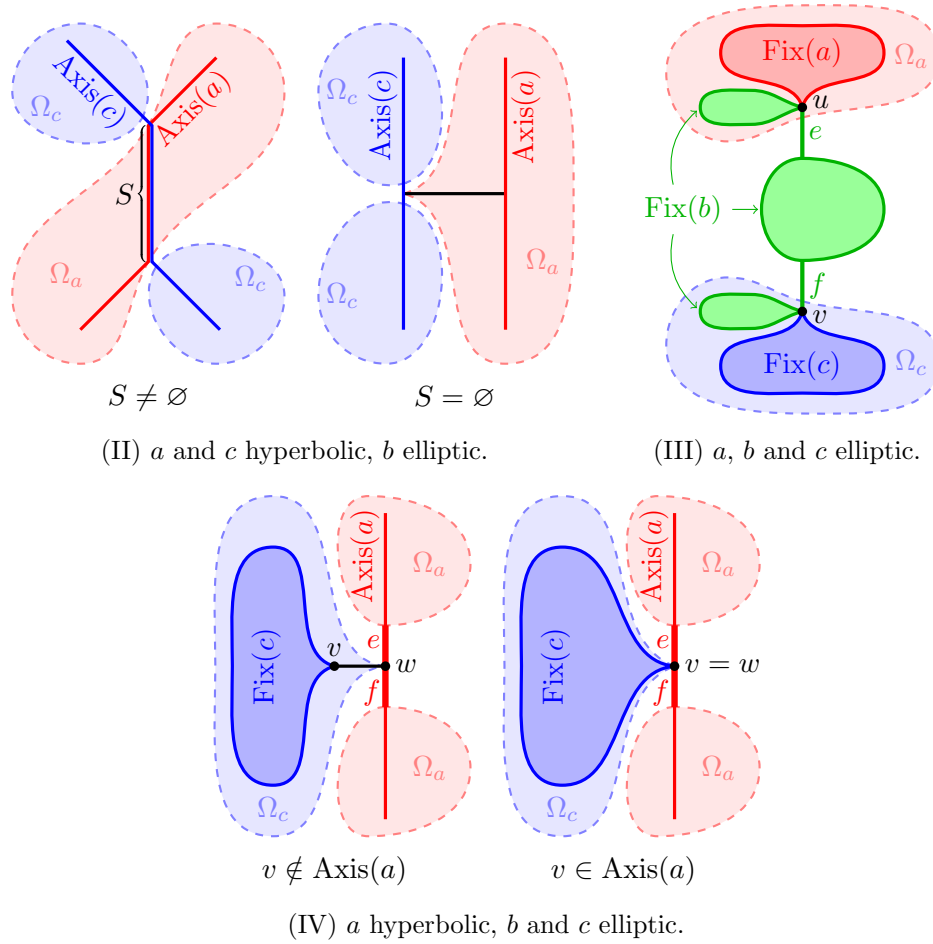
\begin{figure}[ht]
\renewcommand{\thesubfigure}{\Roman{subfigure}}
\begin{subfigure}[b]{0.6\textwidth}
\stepcounter{subfigure}
\centering
\begin{tikzpicture}[scale=0.9]
\begin{scope}
\draw [dashed,thick,red!50,fill=red!10] (0,1) to[out=60,in=180] (1,2.5) to[out=0,in=90] (2,1.5) to[out=-90,in=60] (0,-1) to[out=-120,in=0] (-1,-2.5) to[out=180,in=-90] (-2,-1.5) to[out=90,in=-120] cycle;
\node [red!50] at (-1.2,-1.2) {$\Omega_a$};
\draw [dashed,thick,blue!50,fill=blue!10] (-0.1,1.1) to[out=60,in=0] (-1,2.5) to[out=180,in=90] (-2,1.5) to[out=-90,in=-120] cycle;
\node [blue!50] at (-1.5,1.25) {$\Omega_c$};
\draw [dashed,thick,blue!50,fill=blue!10] (0.1,-1.1) to[out=-120,in=180] (1,-2.5) to[out=0,in=-90] (2,-1.5) to[out=90,in=60] cycle;
\node [blue!50] at (1.5,-1.25) {$\Omega_c$};
\draw [very thick,red] (-1,-2) -- (0,-1) -- (0,1) -- (1,2) node [below left,yshift=3pt,rotate=45] {$\ax(a)$};
\draw [very thick,blue,xshift=1.5pt] (1,-2) -- (0,-1) -- (0,1) -- (-1.25,2.25) node [below right,yshift=3pt,rotate=-45] {$\ax(c)$};
\draw [thick,decorate,decoration={calligraphic brace}] (-0.05,-1) -- (-0.05,1) node [midway,left] {$S$};
\node at (0,-3) {$S \neq \varnothing$};
\end{scope}
\begin{scope}[xshift=4.5cm]
\draw [dashed,thick,red!50,fill=red!10] (-0.75,0) to[out=15,in=180] (0.75,2.5) to[out=0,in=90] (1.75,0) to[out=-90,in=0] (0.75,-2.5) to[out=180,in=-15] (-0.75,0);
\node [red!50] at (1.3,-1) {$\Omega_a$};
\draw [dashed,thick,blue!50,fill=blue!10] (-0.75,0.15) to[out=15,in=0] (-1,2.5) to[out=180,in=90] (-2,1.5) to[out=-90,in=-165] cycle;
\node [blue!50] at (-1.65,1.5) {$\Omega_c$};
\draw [dashed,thick,blue!50,fill=blue!10] (-0.75,-0.15) to[out=-15,in=0] (-1,-2.5) to[out=180,in=-90] (-2,-1.5) to[out=90,in=165] cycle;
\node [blue!50] at (-1.65,-1.5) {$\Omega_c$};
\draw [very thick] (-0.75,0) -- (0.75,0);
\draw [very thick,blue] (-0.75,-2) -- (-0.75,2) node [above left,xshift=2pt,yshift=-0cm,rotate=90] {$\ax(c)$};
\draw [very thick,red] (0.75,-2) -- (0.75,2) node [below left,xshift=-2pt,rotate=90] {$\ax(a)$};
\node at (0,-3) {$S = \varnothing$};
\end{scope}
\end{tikzpicture}
\caption{$a$ and $c$ hyperbolic, $b$ elliptic.}
\label{sfig:P3-heh}
\end{subfigure}
\hfill
\begin{subfigure}[b]{0.35\textwidth}
\centering
\begin{tikzpicture}[scale=0.9]
\draw [dashed,thick,red!50,fill=red!10] (0,1.35) to[out=15,in=-90] (2,2.25) to[out=90,in=0] (0,3) to[out=180,in=90] (-2,2) to[out=-90,in=180] (-1.1,1) to[out=0,in=-165] cycle;
\node [red!50] at (1.6,2.25) {$\Omega_a$};
\draw [dashed,thick,blue!50,fill=blue!10] (0,-1.35) to[out=-15,in=90] (2,-2.25) to[out=-90,in=0] (0,-3) to[out=180,in=-90] (-2,-2) to[out=90,in=180] (-1.1,-1) to[out=0,in=165] cycle;
\node [blue!50] at (1.6,-2.25) {$\Omega_c$};
\draw [ultra thick,green!70!black] (0,1.5) -- (0,0.75) node [midway,right,xshift=-2pt] {$e$};
\draw [ultra thick,green!70!black] (0,-1.5) -- (0,-0.75) node [midway,right,xshift=-2pt] {$f$};
\draw [very thick,red,fill=red!30] (0,1.5) to[out=75,in=-90] (1.2,2.3) to[out=90,in=0] (0,2.7) to[out=180,in=90] (-1.2,2.3) to[out=-90,in=105] (0,1.5);
\node [red] at (0,2.3) {$\fix(a)$};
\draw [very thick,green!70!black,fill=green!40] (0,1.5) to[out=-165,in=0] (-0.75,1.25) to[out=180,in=-90] (-1.5,1.5) to[out=90,in=180] (-0.75,1.75) to[out=0,in=165] (0,1.5);
\draw [very thick,green!70!black,fill=green!40] (0,-1.5) to[out=165,in=0] (-0.75,-1.25) to[out=180,in=90] (-1.5,-1.5) to[out=-90,in=180] (-0.75,-1.75) to[out=0,in=-165] (0,-1.5);
\draw [very thick,green!70!black,fill=green!40] (0,0.75) to[out=0,in=90] (1.2,0) to[out=-90,in=0] (0,-0.75) to[out=180,in=-90] (-0.6,0) to[out=90,in=180] cycle;
\node [green!70!black] (fixb) at (-1.75,0) {$\fix(b)$};
\draw [->,green!70!black] (fixb) to[bend left] (-1.55,1.45);
\draw [->,green!70!black] (fixb) to[bend right] (-1.55,-1.45);
\draw [->,green!70!black] (fixb) to (-0.65,0);
\draw [very thick,blue,fill=blue!30] (0,-1.5) to[out=-75,in=90] (1.2,-2.3) to[out=-90,in=0] (0,-2.7) to[out=180,in=-90] (-1.2,-2.3) to[out=90,in=-105] (0,-1.5);
\node [blue] at (0,-2.3) {$\fix(c)$};
\fill (0,1.5) circle (2pt) node [right,yshift=2pt] {$u$};
\fill (0,-1.5) circle (2pt) node [right,yshift=-2pt] {$v$};
\end{tikzpicture}
\caption{$a$, $b$ and $c$ elliptic.}
\label{sfig:P3-eee}
\end{subfigure}
\bigskip

\begin{subfigure}{\textwidth}
\centering
\begin{tikzpicture}[scale=0.9]
\begin{scope}[xshift=-4.25cm,xscale=-1]
\draw [dashed,thick,blue!50,fill=blue!10] (-0.75,0) to[out=5,in=180] (0.9,2.5) to[out=0,in=90] (1.75,0) to[out=-90,in=0] (0.9,-2.5) to[out=180,in=-5] (-0.75,0);
\node [blue!50] at (0.9,-2.1) {$\Omega_c$};
\draw [dashed,thick,red!50,fill=red!10] (-0.75,0.6) to[out=0,in=0] (-0.5,2.5) to[out=180,in=90] (-2,1.5) to[out=-90,in=180] cycle;
\node [red!50] at (-1.4,1.5) {$\Omega_a$};
\draw [dashed,thick,red!50,fill=red!10] (-0.75,-0.6) to[out=0,in=0] (-0.5,-2.5) to[out=180,in=-90] (-2,-1.5) to[out=90,in=180] cycle;
\node [red!50] at (-1.4,-1.5) {$\Omega_a$};
\draw [very thick] (-0.75,0) -- (0,0);
\draw [very thick,red] (-0.75,-2.25) -- (-0.75,2.25) node [above left,xshift=2pt,yshift=6pt,rotate=90] {$\ax(a)$};
\draw [line width=2.4pt,red] (-0.75,0) -- (-0.75,0.6) node [pos=0.6,left,xshift=3pt] {$e$};
\draw [line width=2.4pt,red] (-0.75,0) -- (-0.75,-0.6) node [pos=0.65,left,xshift=3pt] {$f$};
\draw [very thick,blue,fill=blue!30] (0,0) to[out=15,in=180] (0.9,1.75) to[out=0,in=90] (1.5,0) to[out=-90,in=0] (0.9,-1.75) to[out=180,in=-15] (0,0);
\fill (-0.75,0) circle (2pt) node [right] {$w$} (0,0) circle (2pt) node [above] {$v$};
\node [blue,rotate=90] at (1,0) {$\fix(c)$};
\node at (0,-3) {$v \notin \ax(a)$};
\end{scope}
\begin{scope}[xscale=-1]
\draw [dashed,thick,blue!50,fill=blue!10] (-0.75,0) to[out=20,in=180] (0.9,2.5) to[out=0,in=90] (1.75,0) to[out=-90,in=0] (0.9,-2.5) to[out=180,in=-20] (-0.75,0);
\node [blue!50] at (0.9,-2.1) {$\Omega_c$};
\draw [dashed,thick,red!50,fill=red!10] (-0.75,0.6) to[out=0,in=0] (-0.5,2.5) to[out=180,in=90] (-2,1.5) to[out=-90,in=180] cycle;
\node [red!50] at (-1.4,1.5) {$\Omega_a$};
\draw [dashed,thick,red!50,fill=red!10] (-0.75,-0.6) to[out=0,in=0] (-0.5,-2.5) to[out=180,in=-90] (-2,-1.5) to[out=90,in=180] cycle;
\node [red!50] at (-1.4,-1.5) {$\Omega_a$};
\draw [very thick,red] (-0.75,-2.25) -- (-0.75,2.25) node [above left,xshift=2pt,yshift=6pt,rotate=90] {$\ax(a)$};
\draw [line width=2.4pt,red] (-0.75,0) -- (-0.75,0.6) node [pos=0.6,left,xshift=3pt] {$e$};
\draw [line width=2.4pt,red] (-0.75,0) -- (-0.75,-0.6) node [pos=0.65,left,xshift=3pt] {$f$};
\draw [very thick,blue,fill=blue!30] (-0.75,0) to[out=5,in=180] (0.9,1.75) to[out=0,in=90] (1.5,0) to[out=-90,in=0] (0.9,-1.75) to[out=180,in=-5] (-0.75,0);
\fill (-0.75,0) circle (2pt) node [right] {$v=w$};
\node [blue,rotate=90] at (1,0) {$\fix(c)$};
\node at (0,-3) {$v \in \ax(a)$};
\end{scope}
\end{tikzpicture}
\caption{$a$ hyperbolic, $b$ and $c$ elliptic.}
\label{sfig:P3-hee}
\end{subfigure}

\caption{Some cases in the proof of Proposition~\ref{prop:P3}.}
\label{fig:P3}
\end{figure}

\renewcommand{\thesubsubsection}{\textit{Case~\Roman{subsubsection}\hspace{-1pt}}}

\subsubsection{$b$ is hyperbolic}
Note, first, that there exist $(k,p) \in \Z^2$, with $k \neq 0$, such that $a^kb^p$ fixes $\ax(b)$ pointwise. Indeed, if $a$ is elliptic then $\ax(b) \subseteq \fix(a)$ and we may take $(k,p) = (1,0)$ (see Lemma~\ref{lem:commuting-types}.\ref{it:commtypes-hypell}); and if $a$ is hyperbolic then $\ax(a) = \ax(b)$ and we may choose $k,p \in \Z \setminus \{0\}$ so that $a^k$ and $b^{-p}$ translate $\ax(b)$ with the same translation length and the same direction. Similarly, there exist $(q,m) \in \Z^2$, with $m \neq 0$, such that $b^qc^m$ fixes $\ax(b)$ pointwise. Thus the subgroup $A = \langle a^kb^p,b^qc^m \rangle \leqslant G$ fixes $\ax(b)$ pointwise, so it must be free because $G$ acts on $\mathcal{T}$ with free edge stabilisers. Since $\psi$ is injective on the union  $\langle \alpha,\beta \rangle \cup \langle \beta,\gamma\rangle$ and $\langle \alpha,\beta \rangle \cap\langle \beta,\gamma\rangle=\langle \beta \rangle$ in $A(P_3)$, it follows that $\langle a^kb^p \rangle \cap \langle b^qc^m \rangle = \{1\}$ in $G$. Hence $A$ is not cyclic, implying that $A$ is freely generated by $a^kb^p$ and $b^qc^m$, as required.

\subsubsection{$a$ and $c$ are hyperbolic, $b$ is elliptic} Then $\ax(a) \cup \ax(c) \subseteq \fix(b)$ by Lemma~\ref{lem:commuting-types}.\ref{it:commtypes-hypell}.

Since $\psi$ is injective on $\langle \alpha,\beta \rangle \cup \langle \beta,\gamma\rangle $, we know that $\langle a,b \rangle \cong \Z^2$ and $\langle a,b \rangle \cap \langle c \rangle=\{1\}$ in $G$. Therefore, by 
Lemma~\ref{lem:bounded_intersec_of_axes}, the intersection $S := \ax(a) \cap \ax(c)$ must be bounded. Choose $k,m \in \Z \setminus \{0\}$ so that $a^k\,S \cap S = \varnothing$ and $c^m\,S \cap S = \varnothing$. Let $\pi_a\colon \cT \to \ax(a)$ and $\pi_c\colon \cT \to \ax(c)$ be the closest point projections, let $\Omega_a = \pi_c^{-1}(\pi_c(\ax(a)))$, and let $\Omega_c = \cT \setminus \Omega_a$ (see Figure~\ref{fig:P3}.(\subref{sfig:P3-heh})). We then have $c^{mn}\,\Omega_a \subseteq \Omega_c$, for all $n \in \Z \setminus \{0\}$. Moreover, $\pi_a(\Omega_c)$ consists either of a single point (if $S$ is empty or a point) or of the endpoints of $S$ (if $S$ is a non-trivial interval), implying that $a^{kn}\,\Omega_c \subseteq \Omega_a$, for all $n \in \Z \setminus \{0\}$. It then follows, by version~1 of the Ping-Pong Lemma (Proposition~\ref{prop:PPv1}), that $a^k$ and $c^m$ freely generate a free group, and therefore $\varphi_{k,1,m}|_{\langle \alpha,\gamma \rangle}$ is injective, as required.

\subsubsection{$a$, $b$ and $c$ are elliptic} Then $\fix(a) \cap \fix(b) = \{u\}$ and $\fix(b) \cap \fix(c) = \{v\}$, for some vertices $u,v$ of $\cT$, by  Lemma~\ref{lem:commuting-types}.\ref{it:commtypes-elliptic}, because $G$-stabilisers of edges in $\cT$ are free and $\langle a,b \rangle \cong \langle b,c \rangle \cong \Z^2$. Moreover, we have $u \neq v$, since otherwise $\langle a,b,c \rangle$ would fix a vertex in $\cT$, contradicting our assumption.

Let $e$ and $f$ be the first and the last edges of the geodesic segment $[u,v] \subseteq \fix(b)$ in $\cT$ respectively, in particular, $b\,e=e$ and $b\,f=f$. Let $\Omega_a$ be the connected component containing the vertex $u$, obtained after we remove the interior of the edge $e$ from $\cT$, and let $\Omega_c$ be the connected component containing $v$, after we remove the interior of $f$ from $\cT$ (see Figure~\ref{fig:P3}.(\subref{sfig:P3-eee})). It is easy to see that if $a^n\,\Omega_c \nsubseteq \Omega_a$, for some $n \in \Z$, then $a^n$ must fix $e$, and so $n=0$ since $\langle a^n,b \rangle \leqslant \st_G(e)$ must be both free and free abelian at the same time (by the injectivity of $\psi$ on $\langle \alpha, \beta\rangle$). Therefore, $a^n\,\Omega_c \subseteq \Omega_a$, for all $n \in \Z \setminus \{0\}$. Similarly, we have $c^n\,\Omega_a \subseteq \Omega_c$, for all $n \in \Z \setminus \{0\}$. It then follows from Proposition~\ref{prop:PPv1} that $\varphi_{1,1,1}|_{\langle \alpha,\gamma \rangle}$ is injective, as required.

\subsubsection{$a$ is hyperbolic, $b$ and $c$ are elliptic} Then $\ax(a) \subseteq \fix(b)$ and $\fix(b) \cap \fix(c) = \{v\}$, for some vertex $v$ of $\cT$ which is also the unique common fixed point of $c^n$ and $b$, for any $n \in \mathbb{Z} \setminus \{0\}$.

Let $w$ be the vertex of $\ax(a)$ closest to $v$ (it may happen that $w=v$), and let $e,f$ be the two edges on $\ax(a)$ starting at $w$ and facing in opposite directions. Observe that no non-zero power of $c$ can fix $e$ or $f$ because these edges are already fixed by $b$, because $\langle b,c^n \rangle \cong \Z^2$, for any $n \in \Z\setminus\{0\}$, by the assumptions on $\psi$, and because edge stabilisers are free.

After removing the interiors of the edges $e$ and $f$ from $\cT$ we get three connected components. Let $\Omega_c$ be the connected component containing $w$ (and $v$), and let $\Omega_a$ be the union of the remaining two connected components (each of them will contain an infinite ray of $\ax(a)$), see Figure~\ref{fig:P3}.(\subref{sfig:P3-hee}). It is then easy to see that $a^n \,\Omega_c \subseteq \Omega_a$, for any $n \in \mathbb{Z} \setminus \{0\}$. On the other hand, if $w=v$ then it is, a priori,  possible that $c^n \,\Omega_a \cap \Omega_a \neq \varnothing$, for some $n \neq 0$ (this can only happen if $c^n$ sends $e$ to $f$ or $f$ to $e$).  However, in that case we have $c^t \,\Omega_a \subseteq \Omega_c$, for all $t \in \Z \setminus \{0,\pm n\}$ (as otherwise either $c^{t-n}$ or $c^{t+n}$ would fix $e$ or $f$, which cannot happen as observed above). It follows that there is $m \in \Z \setminus \{0\}$ such that $c^{mn} \, \Omega_a \subseteq \Omega_c$, for all $n \in \Z \setminus \{0\}$. Therefore (by Proposition~\ref{prop:PPv1}) $a$ and $c^m$ freely generate a free group, i.e., $\varphi_{1,1,m}|_{\langle \alpha,\gamma \rangle}$ is injective, as required.

\subsubsection{$a$ and $b$ are elliptic, $c$ is hyperbolic} This is similar to Case~IV. \qedhere
\end{proof}

\section{Embedding \texorpdfstring{$A(P_4)$}{A(P\_4)}} \label{sec:P_4}
We now concentrate on proving  Proposition~\ref{prop:P_4-enhanced} from the Introduction. Throughout this section we will use the standard presentation \eqref{eq:pres-P_4} of $A(P_4)$. 
The following observation is an easy consequence of the normal form theorem for right-angled Artin groups (see \cite[Theorem~3.9]{Green}).

\begin{rem}\label{rem:expected_intersec} The subgroup $\langle \beta ,\gamma\rangle \leqslant A(P_4)$ is free abelian of rank $2$ and the following equalities are true:
\begin{align*}  \alpha \langle \beta,\gamma \rangle \alpha^{-1} \cap  \langle \beta,\gamma \rangle &=
\langle \beta,\alpha\gamma \alpha^{-1}\rangle  \cap  \langle \beta,\gamma \rangle =\langle \beta \rangle, \\
\langle \beta,\gamma \rangle \cap \delta^{-1}\langle \beta,\gamma \rangle \delta &=
\langle \beta,\gamma \rangle \cap  \langle \delta^{-1}\beta\delta,\gamma \rangle = \langle \gamma \rangle, 
\text{ and} \\
\alpha \langle \beta,\gamma \rangle \alpha^{-1} \cap  \delta^{-1}\langle \beta,\gamma \rangle \delta &=
\langle \beta,\alpha\gamma\alpha^{-1} \rangle \cap  \langle \delta^{-1}\beta\delta,\gamma \rangle = \{1\}.
\end{align*}    
\end{rem}

In some cases in order to show the injectivity of a group homomorphism $A(P_4) \to G$ we will apply the following criterion.

\begin{lemma}\label{lem:inj_on-xyz->inj}
Let $G$ be a group and let $\varphi\colon A(P_4) \to G$ be a group homomorphism. Then $\varphi$ is injective if and only if $\varphi(x_1)$, $\varphi(y_1)$ and $\varphi(z_1)$ freely generate a free subgroup of rank $3$ in $G$, where $x_1 = \alpha^{-1}\beta$, $y_1=\beta^{-1}\gamma$ and $z_1=\gamma^{-1}\delta$.
\end{lemma}

\begin{proof}
Let $H = \langle x_1,y_1,z_1 \rangle \leqslant A(P_4)$. By Lemma~\ref{lem:BBfree}, $H$ is the Bestvina--Brady subgroup of $A=A(P_4)$. In particular, $H$ contains the derived subgroup $[A,A]$, of $A$, and it is free, freely generated by $x_1$, $y_1$ and $z_1$.

Suppose, first, that $\varphi$ is injective. Then $\varphi|_H$ is also injective. In particular, since $x_1$, $y_1$ and $z_1$ freely generate a free group, so do $\varphi(x_1)$, $\varphi(y_1)$ and $\varphi(z_1)$, as required.

Conversely, suppose that $\varphi(x_1)$, $\varphi(y_1)$ and $\varphi(z_1)$ freely generate a free subgroup of $G$, and, thus, $\varphi|_H$ is injective, i.e., $\ker(\varphi) \cap H = \{1\}$. Then $\ker(\varphi)$ centralises $H$ in $G$ (because these are disjoint normal subgroups). In particular, we have $\ker(\varphi) \subseteq C_{A}(x_1y_1z_1) = C_{A}(\alpha^{-1}\delta)$. But, from the description of centralisers in right-angled Artin groups due to Servatius \cite{Servatius}, it follows that $C_{A}(\alpha^{-1}\delta) = \langle \alpha^{-1}\delta \rangle = \langle x_1y_1z_1\rangle \subset H$, implying that $\ker(\varphi) = \ker(\varphi) \cap H = \{1\}$, as required.
\end{proof}

In view of Remark~\ref{rem:expected_intersec}, Proposition~\ref{prop:P_4-enhanced} can be restated as follows.

\begin{prop}\label{prop:n=4} Let $G$ be a one-relator group and let $\psi\colon A(P_4) \to G$ be a homomorphism such that $\psi$ is injective on the union 
\begin{equation}\label{eq:union_of_3}
    \alpha \langle \beta, \gamma \rangle \alpha^{-1} \cup \langle \beta, \gamma \rangle \cup \delta^{-1} \langle \beta, \gamma \rangle \delta , 
\end{equation}
of three free abelian subgroups of rank $2$ in $A(P_4)$. Then there exist $k,l,m,n \in \Z\setminus\{0\}$ such that the elements $\psi(\underline{\alpha})^k$, $\psi(\beta)^l$, $\psi(\gamma)^m$ and $\psi(\underline{\delta})^n$ generate a copy of $A(P_4)$ in $G$ in the natural way, where $\underline{\alpha}=\alpha \gamma \alpha^{-1}$ and $\underline{\delta}=\delta^{-1}\beta \delta$ in $G$.   
\end{prop}

\begin{proof} Denote $a=\psi(\alpha)$, $b=\psi(\beta)$, $c=\psi(\gamma)$ and $d=\psi(\delta)$ in $G$. We also set $\aa=aca^{-1}=\psi(\underline{\alpha})$ and $\dd=d^{-1} b d=\psi(\underline{\delta})$.
Arguing by induction on the number $s$ in the sequence~\eqref{eq:HNN} for $G$, we may assume that $s \geq 1$ and that the subgroup $\langle a,b,c,d \rangle \leqslant G$ is not conjugate into $G_{s-1}$. In geometric terms, this means that $\langle a,b,c,d \rangle$ does not fix any point in $\cT$, where $\cT$ is the Bass--Serre tree for the splitting of $G$ as an HNN-extension of $G_{s-1}$.

Suppose the elements $b,c$ both fix some vertex $v$ of $\cT$. If $\aa\,v=v$ then $c\,(a^{-1}\,v)=a^{-1}\,v$, i.e., $a^{-1}\,v \in \fix(c)$. Recall that $\fix(b)$ is invariant under the action of $a^{-1} \in C_G(b)$, by Lemma~\ref{lem:commuting-types}.\ref{it:commtypes-preserved}, so $a^{-1}\,v \in \fix(b)$. Thus $a^{-1}\,v \in \fix(b) \cap \fix(c)=\{v\}$ by Lemma~\ref{lem:commuting-types}.\ref{it:commtypes-elliptic}, hence $a^{-1}\,v=v$, i.e., $v \in \fix(a)$. Similarly, if $\dd$ fixes $v$ then the same must be true for $d$. Therefore, we can further assume that the subgroup $\langle \aa,b,c,\dd \rangle$ does not fix any vertex of $\cT$.

Note that the isometry of $\cT$ induced by $\aa$ has the same type as the one induced by $c$ since these elements are conjugate in $G$; similarly, $\dd$ has the same type as $b$. By symmetry, after swapping $(a,b) \leftrightarrow (d,c)$ if necessary, we may assume that the actions of $\aa$, $b$, $c$ and $\dd$ on $\cT$ fall into one of the following four cases:
\begin{enumerate}[label=\arabic*.]
\item $b$ and $c$ are hyperbolic;
\item $b$ and $c$ are elliptic, $\fix(\aa) \cap \fix(c) = \varnothing = \fix(b) \cap \fix(\dd)$;
\item $b$ and $c$ are elliptic, $\fix(\aa) \cap \fix(c) \neq \varnothing$;
\item $b$ is hyperbolic, $c$ is elliptic.
\end{enumerate}

We now consider each of these cases in order.

\renewcommand{\thesubsubsection}{\textit{Case~\arabic{subsubsection}\hspace{-1pt}}}

\medskip
\subsubsection{$b$ and $c$ are hyperbolic.}~

As observed  above, the elements $\aa$ and $\dd$ will also be hyperbolic.
Since $[\aa,b] = [b,c] = [c,\dd] = 1$, Lemma~\ref{lem:commuting-types} implies that $\ax(\aa) = \ax(b) = \ax(c) = \ax(\dd) = \ell$. There then exist integers $k,l,m,n \in \Z \setminus \{0\}$ such that $\aa^k$, $b^l$, $c^m$ and $\dd^n$ translate $\ell$ in the same direction and with the same translation length. Therefore, the elements $x = \aa^{-k}b^l$, $y = b^{-l}c^m$ and $z = c^{-m}\dd^n$ are elliptic and fix $\ell$ pointwise.

We claim that there exists an integer $r \in \N$ that $x^r$, $y^r$ and $z^r$ freely generate a free subgroup of $G$. Indeed, the group $H = \langle x,y,z \rangle$ fixes an edge in $\cT$, so it is free and in particular hyperbolic. Moreover, by the assumption of injectivity of $\psi$ on the union \eqref{eq:union_of_3} together with Remark~\ref{rem:expected_intersec}, the subgroups $\langle x \rangle$, $\langle y \rangle$ and $\langle z \rangle$ of $H$ are infinite and have pairwise trivial intersections. Lemma~\ref{lem:large powers} now implies the existence of $r \in \N$ as desired.

Now, since $[\aa,b] = [b,c] = [c,\dd] = 1$, the assignment
\[
\varphi(\alpha) = \aa^{kr}, \quad
\varphi(\beta) = b^{lr}, \quad
\varphi(\gamma) = c^{mr} \quad\text{and}\quad
\varphi(\delta) = \dd^{nr}
\]
extends to a well-defined group homomorphism $\varphi\colon A(P_4) \to G$. Moreover, since $\varphi(\alpha^{-1}\beta) = x^r$, $\varphi(\beta^{-1}\gamma) = y^r$ and $\varphi(\gamma^{-1}\delta) = z^r$ freely generate a free subgroup of $G$, it follows from Lemma~\ref{lem:inj_on-xyz->inj} that $\varphi$ is injective, as required.

\medskip
\subsubsection{$b$ and $c$ are elliptic, $\fix(\aa) \cap \fix(c) = \varnothing = \fix(b) \cap \fix(\dd)$}~
\label{case:-ee-}

The subgroups $\langle \aa,b \rangle=a\langle b,c \rangle a^{-1}$ and $\langle b,c \rangle$ are isomorphic to $\Z^2$ in $G$, by the assumptions and Remark~\ref{rem:expected_intersec}, therefore, by Lemma~\ref{lem:commuting-types}, we have $\fix(\aa) \cap \fix(b) = \{u\}$ and $\fix(b) \cap \fix(c) = \{v\}$, for some vertices $u,v$ of  $\cT$. Note that $u \neq v$ because $\fix(\aa) \cap \fix(c) = \varnothing$ in the present case.

Similarly, $\fix(c) \cap \fix(\dd) = \{w\}$, for some vertex $w$ in $\cT$, and $v \neq w$. Note that since $\fix(b)$ (respectively, $\fix(c)$) contains $u$ and $v$ (respectively, $v$ and $w$), it also contains the geodesic segment in $\cT$ joining those two vertices. It follows that $[v,u] \cap [v,w]=\{v\}$ in $\cT$.

Let $e_1, e_2,e_3,e_4$ the first edges of the geodesic segments $[u,v]$, $[v,u]$, $[v,w]$ and $[w,v]$ respectively. Let $\Omega_1$, $\Omega_2$ and $\Omega_3$ be the connected components containing  $u$, $v$ and $w$, respectively, that we obtain after removing the interiors of the edges $e_1,e_2,e_3,e_4$ from $\cT$ (see Figure~\ref{fig:all-elliptic}). By Lemma~\ref{lem:inj_on-xyz->inj}, it is enough to show that $x = \aa^{-1}b$, $y= b^{-1}c$ and $z = c^{-1}\dd$ freely generate a free subgroup of $G$.

\begin{figure}[ht]
\centering
\begin{tikzpicture}
\draw [dashed,thick,yellow!60!black!50!red,fill=yellow!62!red!13] (-3.35,0) to[out=90,in=-90] (-2.25,1.4) to[out=90,in=0] (-4,3) to[out=180,in=90] (-6.2,0) to[out=-90,in=180] (-3.25,-1.75) to[out=0,in=-60] (-3,-0.5) to[out=120,in=-90] cycle;
\node [yellow!60!black!50!red] at (-4.5,-1.25) {$\Omega_1$};
\draw [dashed,thick,yellow!54!green!65!black,fill=yellow!56!green!15] (-0.15,0) to[out=90,in=-90] (-1.2,1.4) to[out=90,in=180] (0,3) to[out=0,in=90] (1.2,1.4) to[out=-90,in=90] (0.15,0) to[out=-90,in=135] (0.4,-0.2) to[out=-45,in=45] (1,-1.5) to[out=-135,in=0] (0,-2) to[out=180,in=-45] (-1,-1.5) to[out=135,in=-135] (-0.4,-0.2) to[out=45,in=-90] cycle;
\node [yellow!54!green!65!black] at (0,-1.75) {$\Omega_2$};
\draw [dashed,thick,green!70!black!50!blue,fill=green!57!blue!12] (3.35,0) to[out=90,in=-90] (2.25,1.4) to[out=90,in=180] (4,3) to[out=0,in=90] (6.2,0) to[out=-90,in=0] (3.25,-1.75) to[out=180,in=-120] (3,-0.5) to[out=60,in=-90] cycle;
\node [green!70!black!50!blue] at (4.5,-1.25) {$\Omega_3$};
\draw [very thick,red,fill=red!30] (-3.5,0) to[out=165,in=0] (-5,0.7) to[out=180,in=90] (-6,0) to[out=-90,in=180] (-5,-0.7) to[out=0,in=-165] (-3.5,0);
\node [red] at (-5,0) {$\fix(\aa)$};
\draw [very thick,yellow!60!black,fill=yellow!50,rotate around={15:(-3.5,0)}] (-3.5,0) to[out=-75,in=90] (-3.25,-1) to[out=-90,in=0] (-3.5,-1.5) to[out=180,in=-90] (-3.75,-1) to[out=90,in=-105] (-3.5,0);
\draw [very thick,yellow!60!black,fill=yellow!50,rotate around={-30:(0,0)}] (0,0) to[out=-75,in=90] (0.25,-1) to[out=-90,in=0] (0,-1.5) to[out=180,in=-90] (-0.25,-1) to[out=90,in=-105] (0,0);
\draw [very thick,yellow!60!black,fill=yellow!50] (-2.7,0) to[out=90,in=180] (-1.75,0.6) to[out=0,in=90] (-0.8,0) to[out=-90,in=0] (-1.75,-0.6) to[out=180,in=-90] cycle;
\draw [ultra thick,yellow!60!black] (-3.5,0) -- (-2.7,0) node [pos=0.6,above,yshift=-3pt] {$e_1$};
\draw [ultra thick,yellow!60!black] (0,0) -- (-0.8,0) node [pos=0.6,above,yshift=-3pt] {$e_2$};
\node [yellow!60!black] (fixb) at (-1.75,-2) {$\fix(b)$};
\draw [->,yellow!60!black] (fixb.west) to[bend left] (-3.1,-1.5);
\draw [->,yellow!60!black] (fixb.east) to[bend right] (-0.7,-1.4);
\draw [->,yellow!60!black] (fixb) to (-1.75,-0.7);
\draw [very thick,green!70!black,fill=green!40,rotate around={-15:(3.5,0)}] (3.5,0) to[out=-105,in=90] (3.25,-1) to[out=-90,in=180] (3.5,-1.5) to[out=0,in=-90] (3.75,-1) to[out=90,in=-75] (3.5,0);
\draw [very thick,green!70!black,fill=green!40,rotate around={30:(0,0)}] (0,0) to[out=-75,in=90] (0.25,-1) to[out=-90,in=0] (0,-1.5) to[out=180,in=-90] (-0.25,-1) to[out=90,in=-105] (0,0);
\draw [very thick,green!70!black,fill=green!40] (2.7,0) to[out=90,in=0] (1.75,0.6) to[out=180,in=90] (0.8,0) to[out=-90,in=180] (1.75,-0.6) to[out=0,in=-90] cycle;
\draw [ultra thick,green!70!black] (3.5,0) -- (2.7,0) node [pos=0.6,above,yshift=-3pt] {$e_4$};
\draw [ultra thick,green!70!black] (0,0) -- (0.8,0) node [pos=0.6,above,yshift=-3pt] {$e_3$};
\node [green!70!black] (fixc) at (1.75,-2) {$\fix(c)$};
\draw [->,green!70!black] (fixc.east) to[bend right] (3.1,-1.5);
\draw [->,green!70!black] (fixc.west) to[bend left] (0.7,-1.4);
\draw [->,green!70!black] (fixc) to (1.75,-0.7);
\draw [very thick,blue,fill=blue!30] (3.5,0) to[out=15,in=180] (5,0.7) to[out=0,in=90] (6,0) to[out=-90,in=0] (5,-0.7) to[out=180,in=-15] (3.5,0);
\node [blue] at (5,0) {$\fix(\dd)$};
\draw [very thick,yellow!60!black!50!red,fill=yellow!62!red!40] (-3.5,0) to[out=75,in=-90] (-2.5,1.5) to[out=90,in=0] (-3.5,2.5) to[out=180,in=90] (-4.5,1.5) to[out=-90,in=105] (-3.5,0);
\node [yellow!60!black!50!red] at (-3.5,1.5) {$\fix(\aa^{-1}b)$};
\draw [very thick,yellow!54!green!65!black,fill=yellow!56!green!45] (0,0) to[out=105,in=-90] (-1,1.5) to[out=90,in=180] (0,2.5) to[out=0,in=90] (1,1.5) to[out=-90,in=75] (0,0);
\node [yellow!54!green!65!black] at (0,1.75) {$\fix(b^{-1}c)$};
\draw [very thick,green!70!black!50!blue,fill=green!57!blue!35] (3.5,0) to[out=105,in=-90] (2.5,1.5) to[out=90,in=180] (3.5,2.5) to[out=0,in=90] (4.5,1.5) to[out=-90,in=75] (3.5,0);
\node [green!70!black!50!blue] at (3.5,1.5) {$\fix(c^{-1}\dd)$};
\fill (-3.5,0) circle (2pt) node [below,xshift=-4pt] {$u$};
\fill (0,0) circle (2pt);
\draw [->] (0,0.6) node [above] {$v$} -- (0,0.2);
\fill (3.5,0) circle (2pt) node [below,xshift=4pt] {$w$};
\end{tikzpicture}
\caption{The configuration in \ref{case:-ee-}.}
\label{fig:all-elliptic}
\end{figure}
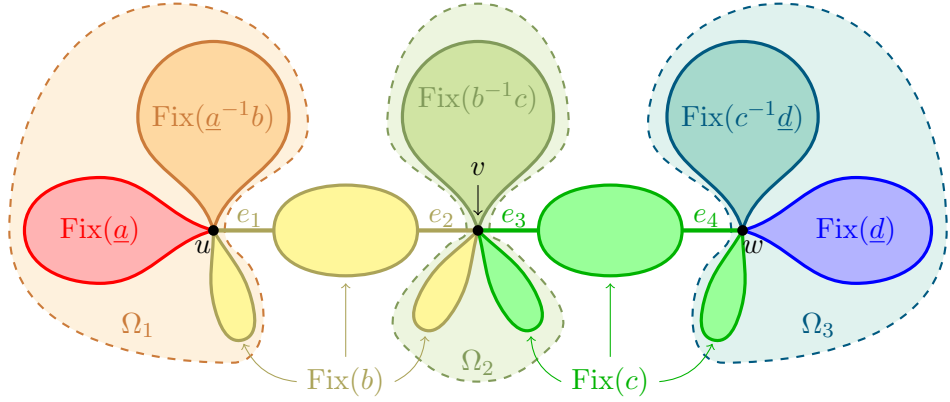

Note, first, that for any $k \in \Z \setminus \{0\}$, we have  $\fix(x^k) \cap \fix(b) = \{u\}$ by Lemma~\ref{lem:commuting-types}, because $\langle x^k,b \rangle=\langle \aa^k,b \rangle \cong \Z^2$. Consequently, $x^k$ does not fix $e_1$, and therefore $x^k\,(\Omega_2 \cup \Omega_3) \subseteq \Omega_1$, for all $k \in \Z \setminus \{0\}$. Similarly, we have $z^k\,(\Omega_1 \cup \Omega_2) \subseteq \Omega_3$, for all $k \in \Z \setminus \{0\}$.

We further claim that $y^k\,(\Omega_1 \cup \Omega_3) \subseteq \Omega_2$, for all $k \in \Z \setminus \{0\}$. Indeed, as in the previous paragraph, we have $\fix(y^k) \cap \fix(b) = \{v\}=\fix(y^k) \cap \fix(c)$, implying that $y^k$ does not fix the edges $e_2$ and $e_3$. Furthermore, $y^k$ cannot send $e_2$ to $e_3$, since in that case we would have both $y^k\,e_2=e_3 \in \fix(c) \cap \fix(b)$ (as $y^k$ preserves $\fix(b)$ setwise by Lemma~\ref{lem:commuting-types}.\ref{it:commtypes-preserved}), contradicting the fact that $\fix(b) \cap \fix(c) = \{v\}$. It follows that $y^k\,\Omega_1 \subseteq \Omega_2$, and a similar argument shows that $y^k\,\Omega_3 \subseteq \Omega_2$; this proves the claim.

Now let $G_1 = \langle x \rangle$, $G_2 = \langle y \rangle$ and $G_3 = \langle z \rangle$. Then, as shown above, we have $g_i\,\Omega_j \subseteq \Omega_i$, for all $i \neq j$ and all $1 \neq g_i \in G_i$. In particular, by version~1 of the Ping-Pong Lemma (Proposition~\ref{prop:PPv1}) we have $\langle x,y,z \rangle = G_1 * G_2 * G_3$, and so $x$, $y$ and $z$ freely generate a free group, as required. Therefore, by Lemma~\ref{lem:inj_on-xyz->inj}, the elements $\aa,b,c,\dd \in G$ generate a copy of $A(P_4)$ in the natural way.

\medskip
\subsubsection{$b$ and $c$ are elliptic and $\fix(\aa) \cap \fix(c) \neq \varnothing$}~
\label{case:+ee-}

As before, using Lemma~\ref{lem:commuting-types} and the assumption that $\psi$ is injective on the subset \eqref{eq:union_of_3}, we can deduce  that $\fix(\aa) \cap \fix(b) = \{u\} = \fix(b) \cap \fix(c)$, for some vertex $u$ in $\cT$. Moreover, using Proposition~\ref{prop:P3}, we may replace $\aa$, $b$ and $c$ with their powers if necessary to assume that the elements $\aa$, $b$ and $c$ generate a copy of $A(P_3)$ in the natural way.

Recall that the injectivity assumptions on $\psi$ imply that $c$ and $\dd$ freely generate a free abelian group of rank $2$ in $G$.
Therefore $\dd=d^{-1}bd$ is elliptic and  $\fix(\dd) \cap \fix(c) = \{v\}$, for some vertex $v$ in $\cT$. Moreover, we have $v \neq u$ because the subgroup $\langle \aa,b,c,\dd \rangle$ does not fix any vertex of $\cT$.

Let $e$ and $f$ be the first and the last edges of  the geodesic segment $[u,v]$ in $\cT$, and let $\Omega_1$ and $\Omega_2$ be the connected components  containing $u$ and $v$, respectively, obtained from $\cT$ after removing the interiors of the edges $e$ and $f$; see Figure~\ref{fig:all-elliptic-amalg}. Let $G_1 = \langle \aa,b,c \rangle \cong A(P_3)$, $G_2 = \langle c,\dd \rangle \cong \Z^2$ and $H = \langle c \rangle \cong \Z$. We claim that $G_1$, $G_2$, $H$, $\Omega_1$ and $\Omega_2$ satisfy the assumptions of version~2 of the Ping-Pong Lemma (Proposition~\ref{prop:PPv2}). In this case it will follow that we have a splitting $\langle \aa,b,c,\dd \rangle \cong G_1 *_H G_2$, implying that $\aa,b,c,\dd$ generate $A(P_4)$ in the natural way.

\begin{figure}[ht]
\centering
\begin{tikzpicture}
\draw [dashed,thick,yellow!60!black,fill=yellow!17] (-2.5,1.5) to[out=60,in=150] (0,2) to[out=-30,in=90] (0.5,1) to[out=-90,in=90] (0.15,0) to[out=-90,in=90] (0.9,-1.5) to[out=-90,in=0] (-0.25,-2.75) to[out=180,in=-120] cycle;
\node [yellow!60!black] at (-1.5,-1) {$\Omega_1$};
\draw [dashed,thick,green!70!black!50!blue,fill=green!57!blue!12] (3.35,0) to[out=90,in=-90] (3,1) to[out=90,in=180] (3.75,2) to[out=0,in=90] (5.5,0) to[out=-90,in=0] (3.75,-2.75) to[out=180,in=-90] (2.6,-1.5) to[out=90,in=-90] cycle;
\node [green!70!black!50!blue] at (4.8,-0.5) {$\Omega_2$};
\draw [very thick,red,fill=red!30,rotate=60] (0,0) to[out=45,in=-90] (0.7,1.5) to[out=90,in=0] (0,2.5) to[out=180,in=90] (-0.7,1.5) to[out=-90,in=105] (0,0);
\node [red] at (-1.2,0.8) {$\fix(a)$};
\draw [very thick,yellow!60!black,fill=yellow!50,rotate=90] (0,0) to[out=165,in=0] (-1.5,0.7) to[out=180,in=90] (-2.5,0) to[out=-90,in=180] (-1.5,-0.7) to[out=0,in=-165] (0,0);
\node [yellow!60!black] at (0,-1.5) {$\fix(b)$};
\draw [very thick,green!70!black,fill=green!40,rotate around={180:(0,0)}] (0,0) to[out=-75,in=90] (0.25,-1) to[out=-90,in=0] (0,-1.5) to[out=180,in=-90] (-0.25,-1) to[out=90,in=-105] (0,0);
\draw [very thick,green!70!black,fill=green!40,rotate around={180:(3.5,0)}] (3.5,0) to[out=-75,in=90] (3.75,-1) to[out=-90,in=0] (3.5,-1.5) to[out=180,in=-90] (3.25,-1) to[out=90,in=-105] (3.5,0);
\draw [very thick,green!70!black,fill=green!40] (2.7,0) to[out=90,in=0] (1.75,0.6) to[out=180,in=90] (0.8,0) to[out=-90,in=180] (1.75,-0.6) to[out=0,in=-90] cycle;
\draw [ultra thick,green!70!black] (0,0) -- (0.8,0) node [pos=0.6,above,yshift=-2pt] {$e$};
\draw [ultra thick,green!70!black] (3.5,0) -- (2.7,0) node [pos=0.6,above,yshift=-2pt] {$f$};
\node [green!70!black] (fixc) at (1.75,2) {$\fix(c)$};
\draw [->,green!70!black] (fixc.east) to[bend left] (3.35,1.55);
\draw [->,green!70!black] (fixc.west) to[bend right] (0.15,1.55);
\draw [->,green!70!black] (fixc) to (1.75,0.7);
\draw [very thick,blue,fill=blue!30,rotate around={-90:(3.5,0)}] (3.5,0) to[out=15,in=180] (5,0.7) to[out=0,in=90] (6,0) to[out=-90,in=0] (5,-0.7) to[out=180,in=-15] (3.5,0);
\node [blue] at (3.5,-1.5) {$\fix(\dd)$};
\fill (0,0) circle (2pt) node [below,xshift=-6pt,yshift=3pt] {$u$};
\fill (3.5,0) circle (2pt) node [right] {$v$};
\end{tikzpicture}
\caption{The configuration in \ref{case:+ee-}.}
\label{fig:all-elliptic-amalg}
\end{figure}
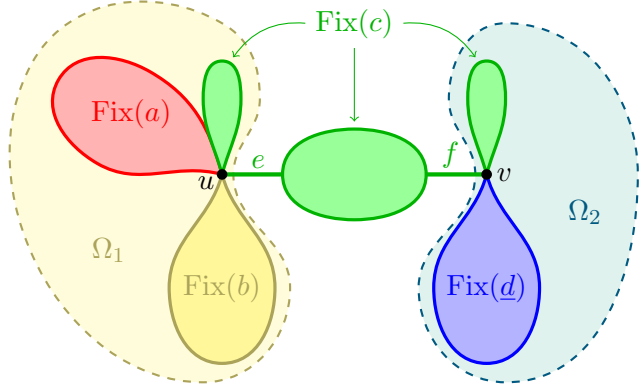

Consider any element $g_2 \in G_2 \setminus H$, so that $g_2 = c^k\dd^l$, for some $k,l \in \Z$ with $l \neq 0$. Then $\langle c,g_2 \rangle=\langle c,\dd^l \rangle \cong \Z^2$ in $G$. By Lemma~\ref{lem:commuting-types}, it follows that $\fix(g_2) \cap \fix(c) = \{v\}$, so $g_2$ does not fix the edge $f$. In particular, we have $g_2\,\Omega_1 \subseteq \Omega_2$, as required. Note that it also follows that $g_2$ does not fix $u$ and therefore $g_2 \notin G_1$; thus $G_1 \cap G_2 \subseteq H$, and since clearly $H \subset G_1$ and $H \subset G_2$, we have $G_1 \cap G_2 = H$ in $G$.

Suppose, for a contradiction, that $g_1\,\Omega_2 \not\subseteq \Omega_1$, for some $g_1 \in G_1 \setminus H$. This means that $g_1\,e=e$. Now, since $c$ fixes $e$ and $\fix(b) \cap \fix(c) = \{u\}$, it follows that $b\,e \neq e$; therefore, by Proposition~\ref{prop:bagh}, the group $C = \st_G(e) \cap \st_G(b\,e)$ is cyclic. However, since $b$ is central in $G_1$, it commutes with both $c$ and $g_1$, implying that $b$ preserves $\fix(c)$ and $\fix(g_1)$ setwise (see Lemma~\ref{lem:commuting-types}.\ref{it:commtypes-preserved}), and therefore $\langle c,g_1 \rangle \subseteq C$. This is impossible, since $g_1 \notin H=\langle c \rangle$ and since $c$ is not a proper power in $G_1=\langle \aa,b,c \rangle \cong A(P_3)$. Thus $g_1\,\Omega_2 \subseteq \Omega_1$, as required.

Finally, it is clear that any element of $H = \langle c \rangle$ fixes $e$ and $f$, and so it preserves $\Omega_1$ and $\Omega_2$ setwise. This proves the claim, implying that we can apply Proposition~\ref{prop:PPv2} to conclude that $\aa,b,c,\dd$ generate a copy of $A(P_4)$ in the natural way.

\medskip
\subsubsection{$b$ is hyperbolic and $c$ is elliptic}~
\label{case:?he?}

In view of Proposition~\ref{prop:P3} we can replace the elements $\aa$, $b$ and $c$ by their powers to assume that the triple $\aa,b,c$ generates a copy of $A(P_3)$ in $G$ in the natural way. 

Lemma~\ref{lem:commuting-types}.\ref{it:commtypes-hypell} implies that $\ax(b) \subseteq \fix(c) \cap \fix(\aa)$. The element $\dd=d^{-1} b d$ is also hyperbolic and  the intersection $S = \ax(b) \cap \ax(\dd)$ is bounded by Lemma~\ref{lem:bounded_intersec_of_axes}, Remark~\ref{rem:expected_intersec} and the injectivity assumption for $\psi$ on the union \eqref{eq:union_of_3}.

Choose  $r \in \N$ so that the translation length $\|b^r\|=r\|b\|$, of $b^r$, is greater than the diameter $\mathrm{diam}(S)$ of $S$; then the same will be true for the translation length of $\dd^r$, because it is conjugate to $b^r$. It follows that  every hyperbolic element of the subgroup $G_1 = \langle \aa,b^r,c \rangle \leqslant G$ has translation length greater than $\mathrm{diam}(S)$.

Let $\pi_{\dd}\colon \cT \to \ax(\dd)$ be the closest point projection. The image  $\pi_{\dd}(\ax(b))$ is a bounded closed segment of $\ax(\dd)$, so there are exactly two oriented edges $e,e' \in \ax(\dd)$ that start at a vertex of $\pi_{\dd}(\ax(b))$ and end at a vertex of $\ax(\dd)\setminus \pi_{\dd}(\ax(b))$. After removing the interiors of the edges $e$ and $e'$ from $\cT$, we get $3$ connected components. Let $\Omega_1$ be the connected component containing $\pi_{\dd}(\ax(b))$, and let $\Omega_2$ be the union of the two other connected components (see Figure~\ref{fig:ell-hyp}).

Let  $G_2 = \langle c,\dd^r \rangle$ and $H = \langle c \rangle$. We claim that $G_1$, $G_2$, $H$, $\Omega_1$ and $\Omega_2$ satisfy the assumptions of version~2 of the Ping-Pong Lemma (Proposition~\ref{prop:PPv2}). In this case it will follow that we have a splitting $\langle a,b^r,c,\dd^r \rangle \cong G_1 *_H G_2$, implying that $\aa,b^r,c,\dd^r$ generate a copy of $A(P_4)$ in $G$ in the natural way, because $G_1 \cong A(P_3)$ and $G_2 \cong \Z^2$.

\begin{figure}[ht]
\centering
\begin{tikzpicture}
\begin{scope}
\draw [dashed,thick,yellow!60!black,fill=yellow!17] (0,0.75) to[out=60,in=180] (1,2.5) to[out=0,in=90] (2,1.5) to[out=-90,in=60] (0,-0.75) to[out=-120,in=0] (-1,-2.5) to[out=180,in=-90] (-2,-1.5) to[out=90,in=-120] cycle;
\node [yellow!60!black] at (-1.2,-1.2) {$\Omega_1$};
\draw [dashed,thick,green!70!black!50!blue,fill=green!57!blue!12] (-0.5,1.25) to[out=60,in=0] (-1,2.5) to[out=180,in=90] (-2,1.5) to[out=-90,in=180] (-1.25,0.5) to[out=0,in=-120] cycle;
\node [green!70!black!50!blue] at (-1.3,0.9) {$\Omega_2$};
\draw [dashed,thick,green!70!black!50!blue,fill=green!57!blue!12] (0.5,-1.25) to[out=-120,in=180] (1,-2.5) to[out=0,in=-90] (2,-1.5) to[out=90,in=0] (1.25,-0.5) to[out=180,in=60] cycle;
\node [green!70!black!50!blue] at (1.3,-0.9) {$\Omega_2$};
\draw [very thick,yellow!60!black] (-1.25,-2) -- (0,-0.75) -- (0,0.75) -- (1.25,2) node [below left,yshift=3pt,rotate=45] {$\ax(b)$};
\draw [very thick,blue,xshift=1.5pt] (1.25,-2) -- (0,-0.75) -- (0,0.75) -- (-1.5,2.25) node [below right,yshift=3pt,rotate=-45] {$\ax(\dd)$};
\draw [line width=2pt,blue,xshift=1.5pt] (0,-0.75) -- (0.5,-1.25) node [below,pos=0.4,xshift=0pt] {$e$};
\draw [blue,ultra thick,->,xshift=1.5pt] (0,-0.75) -- (0.3,-1.05);
\draw [line width=2pt,blue,xshift=1.5pt] (0,0.75) -- (-0.5,1.25) node [above,pos=0.4,xshift=1pt] {$e'$};
\draw [blue,ultra thick,->,xshift=1.5pt] (0,0.75) -- (-0.3,1.05);
\draw [thick,decorate,decoration={calligraphic brace}] (-0.05,-0.75) -- (-0.05,0.75) node [midway,left] {$S$};
\node at (0,-3) {$\mathrm{diam}(S)>0$};
\end{scope}
\begin{scope}[xshift=4.5cm]
\draw [dashed,thick,yellow!60!black,fill=yellow!17] (-0.75,0) to[out=15,in=180] (0.75,2.5) to[out=0,in=90] (1.75,0) to[out=-90,in=0] (0.75,-2.5) to[out=180,in=-15] (-0.75,0);
\node [yellow!60!black] at (1.4,-0.5) {$\Omega_1$};
\draw [dashed,thick,green!70!black!50!blue,fill=green!57!blue!12] (-0.75,0.7) to[out=15,in=0] (-0.75,2.5) to[out=180,in=90] (-2,1.5) to[out=-90,in=-165] cycle;
\node [green!70!black!50!blue] at (-1.6,1.5) {$\Omega_2$};
\draw [dashed,thick,green!70!black!50!blue,fill=green!57!blue!12] (-0.75,-0.7) to[out=-15,in=0] (-0.75,-2.5) to[out=180,in=-90] (-2,-1.5) to[out=90,in=165] cycle;
\node [green!70!black!50!blue] at (-1.6,-1.5) {$\Omega_2$};
\draw [very thick] (-0.75,0) -- (0.75,0);
\draw [very thick,blue] (-0.75,-2) -- (-0.75,2.25) node [below right,xshift=2pt,rotate=-90] {$\ax(\dd)$};
\draw [line width=2pt,blue] (-0.75,-0.7) -- (-0.75,0) node [left,pos=0.4,xshift=2pt] {$e$};
\draw [blue,ultra thick,->] (-0.75,0) -- (-0.75,-0.4);
\draw [line width=2pt,blue] (-0.75,0.7) -- (-0.75,0) node [left,pos=0.4,xshift=3pt] {$e'$};
\draw [blue,ultra thick,->] (-0.75,0) -- (-0.75,0.4);
\draw [very thick,yellow!60!black] (0.75,-2) -- (0.75,2) node [below left,xshift=-1pt,yshift=4pt,rotate=90] {$\ax(b)$};
\fill (-0.75,0) circle (2pt) node [left] {$w$};
\node at (0,-3) {$S = \varnothing$};
\end{scope}
\begin{scope}[xshift=8.5cm]
\draw [dashed,thick,yellow!60!black,fill=yellow!17] (-0.25,0) to[out=30,in=180] (1,2.5) to[out=0,in=90] (1.75,0) to[out=-90,in=0] (1,-2.5) to[out=180,in=-30] (-0.25,0);
\node [yellow!60!black] at (1.4,-0.5) {$\Omega_1$};
\draw [dashed,thick,green!70!black!50!blue,fill=green!57!blue!12] (-0.6,0.5) to[out=30,in=0] (-0.75,2.5) to[out=180,in=90] (-1.75,1.5) to[out=-90,in=-180] cycle;
\node [green!70!black!50!blue] at (-1.35,1.5) {$\Omega_2$};
\draw [dashed,thick,green!70!black!50!blue,fill=green!57!blue!12] (-0.6,-0.5) to[out=-30,in=0] (-0.75,-2.5) to[out=180,in=-90] (-1.75,-1.5) to[out=90,in=180] cycle;
\node [green!70!black!50!blue] at (-1.35,-1.5) {$\Omega_2$};
\draw [very thick,blue] (-0.75,-2) to[out=90,in=-120] (-0.6,-0.5) (-0.6,0.5) to[out=120,in=-90] (-0.75,2.25) node [above right,xshift=-3pt,rotate=-90] {$\ax(\dd)$};
\draw [very thick,yellow!60!black] (1,-2) to[out=90,in=-10] (-0.25,0) (-0.25,0) to[out=10,in=-90] (1,2) node [below left,yshift=3pt,rotate=80] {$\ax(b)$};
\draw [line width=2pt,blue] (-0.6,-0.5) -- (-0.25,0) node [right,pos=0.2,xshift=-2pt] {$e$};
\draw [blue,ultra thick,->] (-0.25,0) -- (-0.45,-0.29);
\draw [line width=2pt,blue] (-0.6,0.5) -- (-0.25,0) node [right,pos=0.1,yshift=1pt,xshift=-1pt] {$e'$};
\draw [blue,ultra thick,->] (-0.25,0) -- (-0.45,0.29);
\fill (-0.25,0) circle (2pt);
\draw [->] (-1,0) node [left,xshift=3pt] {$w$} -- (-0.35,0);
\node at (0,-3) {$S = \{w\}$};
\end{scope}
\end{tikzpicture}
\caption{The possible configurations in \ref{case:?he?}.}
\label{fig:ell-hyp}
\end{figure}
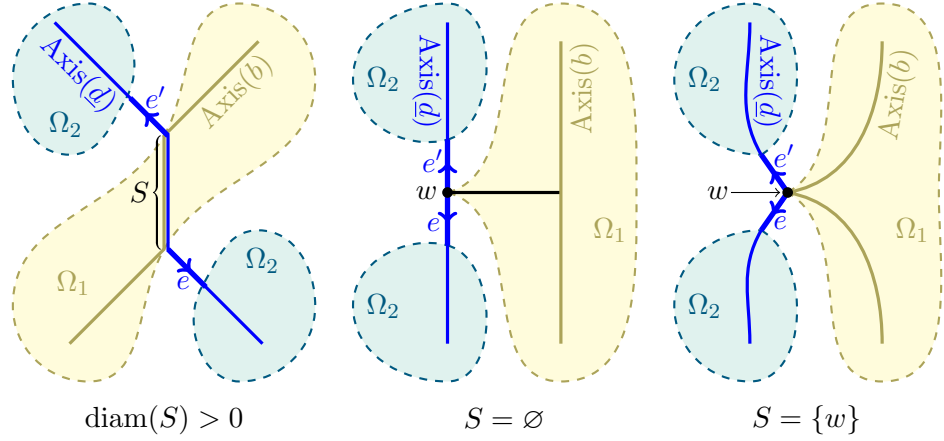

It is clear that any element of $H = \langle c \rangle$ fixes $\ax(b)$ and $\ax(\dd)$ pointwise, and so it preserves $\Omega_1$ and $\Omega_2$ setwise.

Consider any $g_2 \in G_2 \setminus H$. Then 
$g_2 = c^k\dd^{rl}$, for some $k,l \in \Z$ with $l \neq 0$. By Lemma~\ref{lem:commuting-types}, we have $\ax(\dd) \subseteq \fix(c)$, so $g_2$ is hyperbolic with axis $\ax(\dd)$ and translation length $|l| \,\|b^r\|>\mathrm{diam}(S)$. 
It follows that  $g_2\,\Omega_1 \subseteq \Omega_2$, as required. In particular, $g_2$ does not preserve $\ax(b)$ setwise and therefore $g_2 \notin G_1$. Thus $G_1 \cap G_2= H$.

It remains to show that $g_1\,\Omega_2 \subseteq \Omega_1$, for all $g_1 \in G_1 \setminus H = \langle \aa,b^r,c \rangle \setminus \langle c \rangle$. This is a consequence of the following lemma, whose proof is postponed until the end of this section.

\begin{lemma} \label{lem:b.hyp-c.ell}
If $g \in \langle \aa,b^r,c \rangle \setminus \langle c \rangle$ then $g\,\Omega_2 \subseteq \Omega_1$.
\end{lemma}

This lemma  finishes the proof of the claim, so we can apply  Proposition~\ref{prop:PPv2} to deduce that the elements $\aa,b^r,c,\dd^r$ generate a copy  of $A(P_4)$ in $G$ in the natural way, as required.

\medskip
Thus we have considered all possible cases, so the proof of the proposition is complete.
\end{proof}

\begin{proof}[Proof of Lemma~\ref{lem:b.hyp-c.ell}]
Arguing by contradiction, suppose that there is an element $g \in G_1 \setminus H$ such that $g\,\Omega_2 \not\subseteq \Omega_1$.
Since every hyperbolic element of $G_1$ translates $\ax(b)$ by more than $\mathrm{diam}(S)$, we can conclude that $g$ must be elliptic, hence it must pointwise fix $\ax(b) \subseteq \Omega_1$.  The assumption $g\,\Omega_2 \not\subseteq \Omega_1$ now yields that 
\begin{equation}\label{eq:e-e'}
 g\,\{e,e'\} \cap \{e,e'\} \neq \varnothing.   
\end{equation}

First, we claim that $F = \fix(g) \cap \fix(c)$ is equal to $\ax(b)$. Indeed, since $b$ is central in $G_1$ we have $\ax(b) \subseteq \fix(g) \cap \fix(c)$, by Lemma~\ref{lem:commuting-types}.\ref{it:commtypes-hypell}. Hence $F$ is a subtree of $\cT$ containing $\ax(b)$. If $F \neq \ax(b)$ then it must contain a vertex $v$ of degree $\geq 3$ in $F$. Since $G$ acts on $\cT$ with two orbits of oriented edges, there  exists an edge $f \subset F$ incident to $v$ and an element $h \in \st_G(v) \setminus \st_G(f)$ such that $h\,f \in F$. But then $\langle c,g \rangle \subseteq \st_G(f) \cap \st_G(h\,f)$ is cyclic by Proposition~\ref{prop:bagh}; this is impossible, since $g \notin \langle c \rangle$ and $c$ is not a proper power in $G_1=\langle \aa,b^r,c \rangle \cong A(P_3)$. Thus $\fix(g) \cap \fix(c) = \ax(b)$, as claimed. 

It follows that $g\,e \neq e$ and $g\,e' \neq e'$, so, in view of \eqref{eq:e-e'}, either $g\,e=e'$ or $g\,e'=e$. Without loss of generality we can assume that $g\,e=e'$. 
If $\pi_{\dd}(\ax(b))$ contains at least one edge then $\pi_{\dd}(\ax(b))=\ax(b) \cap \ax(\dd)=S$, so it is fixed by $g$ pointwise. In particular, $g$ fixes the start vertices of $e$ and $e'$, the distance between which is positive (as it equals the number of edges in $S$), contradicting the assumption that $g\,e=e'$. Therefore $\pi_{\dd}(\ax(b))$ must consist of a single vertex $w$ in $\cT$, and $e,e'$ are edges of $\ax(\dd)$ starting at $w$ and facing in opposite directions. 

Since $g\,e=e'$, we see that $g\,w=w$, so $\st_G(e) \cap \st_G(e')$ must be cyclic by Proposition~\ref{prop:bagh}. Since both $e$ and $e'$ are edges of $\ax(\dd)$, the pointwise stabiliser $H$, of $\ax(\dd)$, must be cyclic. However, since $\dd=d^{-1}bd$ and $\ax(b)$ is fixed by $\langle \aa,c \rangle$ pointwise, we see that $d^{-1} \langle \aa,c \rangle d \leqslant H$. Recall that $\langle \aa,b,c \rangle \cong A(P_3)$, so the subgroup $\langle \aa,c \rangle$ is isomorphic to the free group of rank $2$. This gives a contradiction with the fact that $H$ is cyclic, finishing the proof of the lemma.
\end{proof}

\section{Free products in strictly ascending HNN-extensions of free groups}\label{sec:asc_HNN}
In the next two sections we will develop some tools for proving Theorem~\ref{thm:one-relator-free-or-snormal} and Theorem~\ref{thm:main} in the case when $\Gamma$ is disconnected. The material from this section will be used  when $G$ is a strictly ascending HNN-extension of a free group.

Let $F$ be a group and let $\varphi\colon F \to F$ be an injective endomorphism. The \emph{ascending HNN-extension of $F$ with respect to $\varphi$} is the group $G$ given by the presentation
\begin{equation}\label{eq:asc_HNN}
    G=\langle F,t \mid tft^{-1}=\varphi(f),\text{ for all }f \in F \rangle.
\end{equation}
If $\varphi(F)$ is a proper subgroup of $ F$, then we will say that $G$ is a \emph{strictly ascending HNN-extension of $F$}.

Suppose that $G$ is given by \eqref{eq:asc_HNN}. Since $t^{n+1}Ft^{-n-1}=\varphi^{n+1}(F) \subseteq \varphi^n(F)=t^n F t^{-n}$, for all $n \in \N\cup \{0\}$, we see that $t$ normalises the subgroup \begin{equation}\label{eq:def_of_N}
N=\bigcap_{n=0}^\infty \varphi^n(F),
\end{equation} whence $\langle N,t\rangle \cong N \rtimes \langle t \rangle$.

On the other hand, the subgroups $(t^{-n} F t^n)_{n \in \N\cup \{0\}}$ form an ascending sequence such that 
\begin{equation}\label{eq:def_of_M}
  M=\bigcup_{n \in \N\cup \{0\}} t^{-n}Ft^n  
\end{equation}
is a normal subgroup of $G$ and $G/M \cong \Z$, where the quotient homomorphism 
\begin{equation}\label{eq:def_of_psi}
\psi\colon G \to \Z ~\text{ is given by } \psi(t)=1 \text{ and }\psi(f)=0, \text{ for all }f \in F.    
\end{equation}

It is easy to see that each element $g \in G$ can be written as the product $t^{-n} f t^m$, for some $m,n \in \N\cup \{0\}$ and $f \in F$. We will say that $g$ is \emph{elliptic} if $m=n$ (equivalently, $g\in \ker\psi= M$); otherwise, we will say that $g$ is \emph{hyperbolic}. This terminology becomes apparent when one considers the natural action of $G$ on its Bass--Serre tree $\cT$. Moreover, in this case it is well-known this action of $G$ on $\cT$ fixes an end (cf. \cite[Proposition~4.13]{MinasyanOsin}).

\begin{notation} \label{not:asc_HNN} 
For the remainder of this section we  assume that $G$ is an ascending 
HNN-extension \eqref{eq:asc_HNN} of a group $F$, and we define $N \leqslant F$, $M \lhd G$ and $\psi\colon G \to \Z$ by \eqref{eq:def_of_N}, \eqref{eq:def_of_M} and \eqref{eq:def_of_psi} respectively.
\end{notation}

\begin{lemma}\label{lem:free_prod_in_asc_HNN} Suppose that $H \leqslant F$ is a non-trivial subgroup such that the subgroup $\langle H, \varphi(F) \rangle$ is naturally isomorphic to the free product $H *\varphi(F)$ in $F$. Then the subgroup $\langle H, N,t \rangle \leqslant G$ is naturally isomorphic to the free product $H*\langle N,t \rangle$ in $G$.    
\end{lemma}

\begin{proof}
Since  $\langle H, \varphi(F) \rangle \cong  H*\varphi(F)$ in $F$ and $\varphi$ is injective, we have 
$\langle \varphi(H), \varphi^2(F) \rangle \cong  \varphi(H)*\varphi^2(F)$ in $\varphi(F)$. Therefore, 
\[\langle H, \varphi(H),\varphi^2(F) \rangle \cong  H*\langle \varphi(H),\varphi^2(F) \rangle\cong H*\varphi(H)*\varphi^2(F).\]
It  follows, by induction on $n \in \N$, that the subgroups $\{\varphi^i(H)\mid 0 \le i \le n-1\}$ and the subgroup $\varphi^n(F)$ freely generate their free product in $F$. Since $N \subseteq \varphi^n(F)$, for all $n \in \N$, it is easy to see (e.g., by using normal forms in free products) that the subgroups $\{ \varphi^n(H) \mid n \geq 0 \}$  together with $N$ generate the free product $\displaystyle \mathop{\ast}_{i=0}^\infty \varphi^i(H) * N$ in $F$.

Finally, we claim that the homomorphism $\chi\colon H * \langle N,t \rangle \to G$, where the restrictions of $\chi$ to $H$ and $\langle N,t \rangle$ are the inclusions of these subgroups into $G$, is injective. Indeed, let $g \in H * \langle N,t \rangle$ be a non-trivial element, and suppose, for contradiction, that $\chi(g) = 1$. Then we have an expression
\[
g = g_0 t^{\alpha_0} h_1 g_1 t^{\alpha_1} \cdots h_n g_n t^{\alpha_n},
\]
for some $n \geq 1$, $g_0,\ldots,g_n \in N$, $\alpha_0,\ldots,\alpha_n \in \Z$ and some $h_1,\dots,h_n \in H$, such that $h_i \neq 1$ for $i=1, \dots, n$, and for each $0 < i < n$ we have either $g_i \neq 1$ or $\alpha_i \neq 0$. Note that we have $\sum_{i=0}^n \alpha_i = \psi(\chi(g))  = 0$, where $\psi\colon G \to \Z$ is defined in \eqref{eq:def_of_psi}.
Since $t$ normalises $N$, we can write \[g = g_0 t^{\gamma_1} h_1 t^{-\gamma_1} g_1' \cdots t^{\gamma_n} h_n t^{-\gamma_n} g_n',\] where $\gamma_i = \alpha_0 + \cdots + \alpha_{i-1}$, for $i=1, \dots, n$, and $g'_1,\dots,g_n' \in N$ are conjugates of $g_1,\dots,g_n$ (respectively) in $G$. After replacing $g$ with $t^m g t^{-m}$, for $m \in\N$ large enough, we may assume that $\gamma_i \geq 0$,  for each $i$, and therefore
\begin{equation} \label{eq:g-expression}
\chi(g) = g_0'  \varphi^{\gamma_1}(h_1)  g_1'' \cdots \varphi^{\gamma_n}(h_n)  g_n'' ~\text{ in } G,
\end{equation}
for some $g_0',g_1'',\dots,g_n'' \in N$ that are conjugate to $g_0,g_1',\dots,g_n'$ (respectively) in $G$. Now, since $h_i \neq 1$, for all $i=1,\dots,n$, and for each $1 \le  i <n$ we have either $g_i \neq 1$ or $\alpha_i \neq 0$, it follows that whenever $g_i'' = 1$, for some $i \in \{1,\ldots,n-1\}$, we have $\gamma_i \neq \gamma_{i+1}$. Therefore, after possibly removing the elements $g_i''$ that are trivial in $N$, the expression \eqref{eq:g-expression} for $\chi(g)$ is reduced as a word in the free product $\displaystyle \mathop{\ast}_{i=0}^\infty \varphi^i(H) * N$. Since the latter free product embeds into $F$ and $n \geq 1$, this reduced word  represents a non-trivial element of $F \leqslant G$, contradicting our assumption. Therefore $\chi$ is injective, and the lemma is proved.
\end{proof}

\begin{lemma}\label{lem:red_to_stable_letter}  Given arbitrary $h \in F$ and $l \in\N$, let $\widetilde\varphi\colon F \to F$ be the injective endomorphism defined by  $\widetilde\varphi(f)=h \varphi^l(f) h^{-1}$, for all $f \in F$, and let $\widetilde G$ be the corresponding ascending HNN-extension of $F$, i.e., 
\begin{equation*}\label{eq:tilde_G}
\widetilde G=\langle F, \widetilde{t} \mid \widetilde t f \widetilde t^{-1}=\widetilde\varphi(f), \text{ for all } f \in F \rangle.    
\end{equation*}
If $g=ht^l \in G$ and $J=\langle F,g \rangle \leqslant G$ then 
\begin{enumerate}[label=\rm(\roman*)]
    \item \label{it:redstable-J} $J=\psi^{-1}(l\Z)$, where $\psi\colon G \to \Z$ is defined by \eqref{eq:def_of_psi};
    \item \label{it:redstable-iso} there is an isomorphism  $\theta\colon \widetilde G \to J$ such that $\theta$ restricts to the identity map on $F$ and $\theta(\hspace{1pt}\widetilde{t}\hspace{1pt})=g$. 
\end{enumerate}
\end{lemma}

\begin{proof}
For part~\ref{it:redstable-J}, note that for every $m\in \N\cup\{0\}$ we have $h \in t^{-m}Ft^m$, whence
\[t^{-n} F t^n \subseteq t^{-ln} F t^{ln}=g^{-n}Fg^n \subseteq J,~\text{ for all }n \in \N\cup\{0\}.\] This implies that $M=\ker\psi \subseteq J$. Since $\psi(J)=l\Z$, we can conclude that $J=\psi^{-1}(l\Z)$.

The proof of part~\ref{it:redstable-iso} is left as an easy exercise for the reader.    
\end{proof}

\begin{defn}\label{def:principal_elt} Under Notation~\ref{not:asc_HNN}, suppose that $A \leqslant G$ is a subgroup such that $\psi(A)=l\Z$, for some $l \in \N$. Then any element $g \in A$ with $\psi(g)=l$ will be called a \emph{principal element of $A$}. Note that $g=t^{-n} h t^m$, for some $m,n \in \N\cup \{0\}$ and $h \in F$, with $m-n=l$.    
\end{defn}

Observe that when the base group $F$ is free, the natural action of $G$, given by \eqref{eq:asc_HNN}, on its Bass--Serre tree $\cT$ has free vertex and edge stabilisers and fixes an end of $\cT$. Therefore  Lemma~\ref{lem:fixed_end} is applicable in this case.

\begin{lemma}\label{lem:asc_HNN-principal} Suppose that $F$ is a free group and $A \leqslant G$  has non-trivial centre and contains a principal element $g=ht^l$, for some $h \in F$ and $l \in \N$. Set $P=h\varphi^l(F)h^{-1} \leqslant F$. If a subgroup $H \leqslant  F$ is such that  $\langle H,P \rangle \leqslant G$ is naturally isomorphic to the free product $H * P$ in $F$ then the subgroup $\langle H,A \rangle $ is naturally isomorphic to $ H * A$ in $G$.    
\end{lemma}

\begin{proof} In view of Lemma~\ref{lem:red_to_stable_letter}, $\psi^{-1}(l\Z)$ is isomorphic to an ascending HNN-extension $\widetilde G$ of $F$, and $g$ is the stable letter $\widetilde t$ in it. Since $A \subseteq \psi^{-1}(l\Z)$, we can replace $\varphi$ with $\widetilde \varphi$ and $G$ with $\widetilde G$ to further assume that $A$ contains the stable letter $t$ of the ascending HNN-extension $G$, given by \eqref{eq:asc_HNN} (so that $l=1$, $h=1$ and $P=\varphi(F)$). 

Since $A$ has non-trivial centre $Z$,  Lemma~\ref{lem:fixed_end} implies that there exists a hyperbolic element $z=t^{-p}at^q \in Z$, where  $a \in F$ and $p,q \in \N\cup \{0\}$ are such that $k=q-p \in \N$. Since $t \in A$ and $z$ is central in $A$, we see that $z=t^ka$ and $a \in A \cap F$ commutes with $t$.
Set $B=A \cap F$, then $A=\langle B ,t \rangle$, $a \in B$ and 
\[B=z^n Bz^{-n}=t^{kn} a^n B a^{-n} t^{-kn}=\varphi^{kn}(B),~ \text{  for all } n \in \N\cup\{0\}.\] Therefore $B \subseteq N$, where $N \leqslant F$ is defined by \eqref{eq:def_of_N}.

By the assumptions, the subgroup $\langle H,P \rangle$ splits as the free product $H * P$ in $F$, where $P=\varphi(F)$. Therefore we can apply Lemma~\ref{lem:free_prod_in_asc_HNN} to conclude that the subgroup  $\langle H,N,t \rangle \leqslant G$ is naturally isomorphic to the free product $H * \langle N,t \rangle$. Since $A=\langle B,t \rangle \leqslant \langle N,t \rangle$, we have a natural isomorphism between $\langle H, A \rangle$ and  $H*A$ in $G$, as required.    
\end{proof}

\begin{lemma}\label{lem:asc_HNN-tot_ell} Suppose that $F$ is a free group and  $A \leqslant G$ is a subgroup with non-trivial centre satisfying $\psi(A)=\{0\}$.  If some elements $a \in F \cap A$ and $f \in F$ do not commute in $F$ then the subgroup $\langle f,A \rangle $ is naturally isomorphic to $ \langle f \rangle * A$ in $G$.  
\end{lemma}

\begin{proof} By the assumptions,  $A \subseteq M=\ker\psi$, thus $A$ consists of elliptic elements only, so $A$ must be locally cyclic by Lemma~\ref{lem:fixed_end}.  Suppose that $\langle f,A \rangle \not\cong \langle f \rangle*A$ in $G$. Then there exist $k \ge 1$, $a_1,a_2,\dots,a_k \in A\setminus\{1\}$ and $i_1,\dots,i_k \in \Z\setminus\{0\}$ such that
\begin{equation}\label{eq:not_a_free_prod}
 a_1f^{i_1}a_2f^{i_2} \dots a_kf^{i_k}=1 \text{ in } G.   
\end{equation}
Now, since $A$ is locally cyclic, there exists an element $b \in A$ such that $a,a_1,\dots,a_k \in \langle b \rangle$. Moreover, in view of \eqref{eq:def_of_M},  $b \in t^{-n} F t^n$, for some $n \in \N$. 

Recall that $t^{-n} F t^n$ is a free group containing $F$. Equation~\eqref{eq:not_a_free_prod} means that $\{b,f\}$ is not a free generating set of the subgroup $\langle b,f \rangle \leqslant t^{-n} F t^n$, hence this subgroup must be cyclic (for example, by \cite[Proposition~I.2.7]{LyndonSchupp}). In particular, $b$ must commute with $f$, whence $a$ must commute with $f$, contradicting the choice of $f$. This contradiction completes the proof.        
\end{proof}

\begin{lemma} \label{lem:existence_of_f} Suppose that $F$ is a non-abelian finitely generated free group and $P_1,\dots, P_s \leqslant F$ are finitely generated subgroups of infinite index in $F$. Then there exists an element $f \in F \setminus\{1\}$ such that the subgroup $\langle f,P_i \rangle$ is naturally isomorphic to the free product $\langle f \rangle*P_i$ in $F$, for each $i=1,\dots,s$.    
\end{lemma}

\begin{proof} Recall that free groups are hyperbolic (in the sense of Gromov) and finitely generated subgroups are quasiconvex (see \cite[Proposition~2]{Short}). Therefore, by \cite[Lemma~4.6]{Min-some_props} there exists an infinite order element $g \in F$ such that $\langle g \rangle \cap P_i=\{1\}$, for each $i=1,\dots,s$. Now we can apply \cite[Theorem~5]{Min-res_props} (with $s=1$, $g_1=g$ and $x_1=1$) to find a sufficiently large $n \in \N$ such that for the element $f=g^n$ and for every $i=1,\dots,s$ the subgroup $\langle f,P_i\rangle$ naturally splits as the free product $\langle f \rangle *P_i$ in $F$.    
\end{proof}

\begin{rem}
The assumption that $F$ is finitely generated in Lemma~\ref{lem:existence_of_f} is not necessary.  Indeed, if $F$ is not finitely generated and $\mathcal{B}$ is an (infinite) free basis for $F$, then only finitely many elements of $\mathcal{B}$ appear in the words defining the generators of the $P_i$. It follows that there exists a free splitting $F = H * K$, with $H$ finitely generated, such that $P_i \leqslant H$ for each $i$, and we could simply take $f$ to be any non-trivial element of $K$.
\end{rem}

\begin{prop} \label{prop:HNN-free-products}
Let $G$ be a strictly ascending HNN-extension of a finitely generated non-abelian free group $F$ with respect to a non-surjective injective endomorphism $\varphi\colon F \to F$. Suppose that  $A_1,\dots,A_s \leqslant G$  are subgroups with non-trivial centres. Then there exists an elliptic element of infinite order $f \in G$  such that for each $i=1,\dots,s$ the subgroup $\langle f,A_i \rangle \leqslant G$ is naturally isomorphic to the free product $\langle f \rangle * A_i$.
\end{prop}

\begin{proof}
Let us assume that $G$ has presentation \eqref{eq:asc_HNN}, and let $\psi\colon G \to \Z$ be defined by \eqref{eq:def_of_psi}. After re-labelling, we can suppose that, for some $q \in \{0,\dots,s\}$, the subgroups $A_1,\dots,A_q$ have non-trivial images under $\psi$ and the subgroups $A_{q+1},\dots,A_s$ have trivial $\psi$-images. For $i=1,\dots,q$, choose a principal element $g_i=t^{-n_i}h_it^{m_i} \in A_i$,  where $m_i-n_i=l_i \in \N$ and $h_i \in F$. For $j=q+1,\dots,s$, choose any non-trivial element $a_j \in A_j$; then $a_j=t^{-n_j} h_j t^{n_j}$, for some $n_j \in \N\cup\{0\}$ and $h_j \in F \setminus\{1\}$.

Recall that $t^k h t^{-k}=\varphi^k(h) \in F$, for all $k \in \N\cup\{0\}$ and all $h \in F$. Therefore, after setting $n=\max\{n_i \mid i=1,\dots,s\}$ and replacing $A_1,\dots,A_s$ with their conjugates by $t^n$, we can assume that $g_i=h_it^{l_i}$, for $i=1,\dots,q$, and $a_j \in (A_j \cap F)\setminus\{1\}$, for $j=q+1,\dots,s$.

Denote $P_i=h_i \varphi^{l_i}(F) h_i^{-1} \leqslant F$, for $i=1,\dots,q$, and $P_j=\langle a_j \rangle \leqslant F$, for $j=q+1,\dots,s$. Recall that $\varphi(F)$ is a proper subgroup of $F$ and $F$ is a free group of finite rank $r \ge 2$. So each $P_i$ is a proper subgroup of $F$ of rank at most $r$, therefore $|F:P_i|=\infty$ by Schreier index formula \cite[Propostion~I.3.9]{LyndonSchupp}. Hence we can use Lemma~\ref{lem:existence_of_f} to find an infinite order element $f \in F$ such that $\langle f,P_i \rangle \cong \langle f \rangle* P_i$, for each $i=1,\dots,s$. In particular, $f$ does not commute with $a_j$, when $q+1\le j\le s$. Therefore, according to Lemmas~\ref{lem:asc_HNN-principal} and \ref{lem:asc_HNN-tot_ell}, for each $i=1,\dots,s$ the subgroup $\langle f,A_i \rangle$ naturally splits as the free product $\langle f \rangle*A_i$ in $G$.
\end{proof}

\section{Free products in one-relator groups}\label{sec:one-rel}
In this section we give a coarse classification of one-relator groups (Proposition~\ref{prop:1-rel_ah-refinement}), which allows us  to prove Theorem~\ref{thm:one-relator-free-or-snormal} by combining the results about strictly ascending HNN-extensions of free groups obtained in Section~\ref{sec:asc_HNN} together with known properties of acylindrically hyperbolic groups.

Recall that a group $G$ is said to be \emph{acylindrically hyperbolic} if it admits a non-elementary acylindrical action on a Gromov hyperbolic geodesic metric space. Acylindrically hyperbolic groups were defined by Osin in \cite{Osin} and we refer the reader to this manuscript for the detailed definitions and examples. For the purposes of the current paper we can use this concept as a black box.

The next proposition is essentially a refinement of a statement about acylindrical hyperbolicity of one-relator groups given by the first author and Osin in \cite{MinasyanOsin}, using recent progress achieved by Genevois and Horbez \cite{GenevoisHorbez} for free-by-cyclic groups. Recall that a group $G$ is said to \emph{commensurate a subgroup} $H\leqslant G$ if $H \cap gHg^{-1}$ has finite index in $H$, for all $g \in G$. When $H=\langle a \rangle$ is cyclic this means that for each $g \in G$ there exist $k,l \in \Z\setminus \{0\}$ such that $g a^k g^{-1}=a^l$ in $G$.

\begin{prop}\label{prop:1-rel_ah-refinement}
Let $G$ be a one-relator group with presentation \eqref{eq:1-rel_pres}. If $k \ge 3$ then $G$ is acylindrically hyperbolic. 

If $k=2$ then at least one of the following statements is true:
\begin{enumerate}[label=\rm(\alph*)]
\item\label{it:comm-cyc} there is an infinite order element $a \in G$ such that $G$ commensurates the cyclic subgroup $\langle a \rangle$;
    \item\label{it:st_asc_HNN} $G$ is isomorphic to a strictly ascending HNN-extension of a finitely generated non-abelian free group $F$ with respect to some injective non-surjective endomorphism $\varphi\colon F \to F$;
    \item\label{it:AH} $G$ is acylindrically hyperbolic.
\end{enumerate}
Moreover, in case \ref{it:comm-cyc} $G$ is a generalised Baumslag--Solitar group. 
\end{prop}

\begin{proof} If $k \ge 3$ then $G$ is acylindrically hyperbolic by \cite[Corollary~2.6]{MinasyanOsin}. Therefore we may further assume that $k=2$. In this case, by \cite[Proposition~4.21]{MinasyanOsin}, at least one of the following holds:
\begin{enumerate}[label=(\roman*)]
\item \label{it:MO-AH} $G$ is acylindrically hyperbolic;
\item \label{it:MO-snormal} $G$ is an HNN-extension of a  one-relator group $H = \langle a,b \mid r \rangle$ with associated infinite cyclic subgroups $\langle a \rangle$ and $\langle b \rangle$ that have non-trivial intersection in $H$ (in this case $G$ commensurates $\langle a \rangle$);
\item \label{it:MO-mappingtorus} $G$ is an ascending HNN-extension of a finitely generated free group.
\end{enumerate}
Thus to prove the lemma it remains to consider cases \ref{it:MO-snormal} and \ref{it:MO-mappingtorus}. If \ref{it:MO-snormal} holds then we are 
obviously in case \ref{it:comm-cyc}, so assume \ref{it:MO-mappingtorus}, i.e., $G$ is an ascending HNN-extension of a finitely generated free group $F$ with respect to an injective endomorphism $\varphi\colon F \to F$. If
$F$ is trivial then $G \cong \Z$, and if
$F$ is infinite cyclic then $G$ is isomorphic to a solvable Baumslag--Solitar group $\langle a,t \mid tat^{-1}=a^l\rangle$, for some $l \in \Z\setminus\{0\}$. In either case $G$ commensurates an infinite cyclic subgroup, so \ref{it:comm-cyc} holds. 

Thus we can further assume that $F$ is non-abelian. 
If $\varphi\colon F \to F$ is not surjective then we are in case \ref{it:st_asc_HNN}. Therefore, we may suppose that $\varphi$ is an automorphism of $F$, and so $G \cong F \rtimes_\varphi \Z$ is a free-by-cyclic group (where $F$ is finitely generated and non-abelian).
If the image of $\varphi$ has infinite order in the outer automorphism group $\operatorname{Out}(F)$, then $G$ is acylindrically hyperbolic by \cite[Corollary~1.5]{GenevoisHorbez}, so we are in case \ref{it:AH}. Otherwise, some power of $\varphi$ is an inner automorphism of $F$, implying that $G$ has a finite-index normal subgroup $H$ splitting as the direct product $F \times \langle a\rangle$, for some infinite order element $a \in G$. Since $F$ is non-abelian, $\langle a \rangle$ is the centre of $H$ and so it is characteristic in $H$. It follows that $\langle a \rangle \lhd G$ and we are in case \ref{it:comm-cyc}. This finishes the consideration of case \ref{it:MO-mappingtorus}.

Finally, suppose that \ref{it:comm-cyc} holds and let us show that $G$ is a generalised Baumslag--Solitar group (essentially, this was  observed in \cite[Theorem~3.2]{Bu-Kr}). If the defining relator $W$ from \eqref{eq:1-rel_pres} is a proper power then $G$ is hyperbolic by Newman's Spelling theorem \cite[Theorem~IV.5.5]{LyndonSchupp} and is not virtually cyclic (for example, because it contains free subgroups of rank $2$ \cite[Theorem~1]{Ree-Men}). On the other hand, it is well-known that the maximal subgroup commensurating an infinite cyclic subgroup in a hyperbolic group is virtually cyclic (see \cite[Lemmas~1.16 and 1.17]{Olsh}), which contradicts \ref{it:comm-cyc}. Therefore $W$ cannot be a proper power, hence $G$ has cohomological dimension at most $2$ by a theorem of Lyndon \cite[Corollary~11.2]{Lyndon}. We can now apply a result of P.~Kropholler \cite[Theorem~C]{PKr} stating that $G$ must be a generalised Baumslag--Solitar group.
\end{proof}

We are now ready to prove Theorem~\ref{thm:one-relator-free-or-snormal} from the Introduction. Our argument combines Propositions~\ref{prop:1-rel_ah-refinement} and \ref{prop:HNN-free-products} with the work of Newman \cite{Newman-thesis} on one-relator groups with torsion and known properties of acylindrically hyperbolic groups obtained by Abbott and Dahmani \cite{AbbottDahmani} and Osin \cite{Osin}.

\begin{proof}[Proof of Theorem~\ref{thm:one-relator-free-or-snormal}]
If $G$ is a strictly ascending HNN-extension of a finitely generated non-abelian free group then the desired result follows from Proposition~\ref{prop:HNN-free-products}.
Thus, in view of Proposition~\ref{prop:1-rel_ah-refinement}, we can assume that $G$ is acylindrically hyperbolic. 

Let us show that if $N \lhd G$ is a finite normal subgroup then $N=\{1\}$. Indeed, suppose there is $h \in N \setminus \{1\}$. Then $\mathrm{C}_G(h)$ has finite index in $G$ and $G$ is a one-relator group with torsion. By a result of Newman \cite[Corollary~2.3.4]{Newman-thesis}, $C_G(h)$ must be cyclic, hence $G$ will be virtually cyclic. But this contradicts the fact that $G$ contains non-abelian free subgroups (see \cite[Theorem~1.3.11]{Newman-thesis}). Thus $G$ cannot have any non-trivial finite normal subgroups.

After renumbering we can assume that the subgroups $A_1,\dots,A_q$ are cyclic and the subgroups $A_{q+1},\dots,A_s$ are non-cyclic, for some $q \in \{0,\dots,s\}$. If $q=s$, i.e., every $A_i$ is cyclic, then by \cite[Theorem~2.3]{AbbottDahmani}, there exists an infinite order element $f \in G$ such that $\langle f,A_i \rangle \cong \langle f \rangle * A_i$, for all $i=1,\dots,s$, so we are done.

Therefore we can assume that $q<s$, so that $A_s$ is not cyclic. For each $j=q+1,\dots,s$, let $Z_j \neq \{1\}$ denote the centre of $A_j$. Since $A_j \leqslant C_G(Z_j)$, $G$ must be torsion-free by the result of Newman mentioned above. It follows that $A_j$ is not virtually cyclic (because torsion-free virtually cyclic groups are cyclic) and its centre $Z_j$ is infinite, for every $j=q+1,\dots,s$.

Observe that if $G$ acts acylindrically on a hyperbolic geodesic metric space then the restriction of this action to $A_j$, $q+1\le j \le s$, must have bounded orbits (i.e., $A_j$ will be elliptic with respect to this action). Indeed, otherwise, since $A_j$ is not virtually cyclic it must be acylindrically hyperbolic itself by \cite[Theorem~1.1]{Osin}, which is impossible because $A_j$ has infinite centre (see \cite[Corollary~7.3]{Osin}).

By \cite[Theorem~1.2]{Osin}, $G$ admits a non-elementary co-bounded acylindrical action on a hyperbolic geodesic metric space $(X,\mathrm{d}_X)$. By \cite[Lemma~2.4 and the proof of Theorem~2.3]{AbbottDahmani} there is a non-elementary acylindrical action of $G$ on a different hyperbolic geodesic metric space $(Y,\mathrm{d}_Y)$ such that the cyclic subgroups $A_1,\dots,A_q$ all act elliptically on $Y$. On the other hand, the restriction of the action to the subgroups $A_{q+1},\dots,A_s$ must also be elliptic, as observed above. Therefore we can apply \cite[Theorem~0.1]{AbbottDahmani} claiming that there exists an infinite order element $f \in G$ such that $\langle f,A_i \rangle \cong \langle f \rangle * A_i$ in $G$, for all $i=1,\dots,s$. So the proof of the theorem is complete.
\end{proof}

\begin{lemma}\label{lem:GBS->non-cyc_sbps_contain_comm_cyclic} Let $G$ be a generalised Baumslag--Solitar group $G$. Then $G$ commensurates an infinite cyclic subgroup $\langle a \rangle \leqslant G$. 
If a subgroup $B \leqslant G$ is not free then $B \cap \langle a \rangle \neq \{1\}$.    
\end{lemma}
\begin{proof} By definition, $G$ is torsion-free and admits an action on a tree $\cT$ with infinite cyclic vertex and edge stabilisers. Let $\langle a \rangle$ be the stabiliser of any vertex $v$ in $\cT$. Then it is easy to see that $G$ commensurates $\langle a \rangle$ (because every edge stabiliser has finite index in the stabilisers of its endpoints). 

Let $B \leqslant G$ be a non-free subgroup. Then, by the well-known fact that a group acting on a tree freely must be free (see, for example, \cite[Theorem~I.8.2]{DD}), there must exist an elliptic element $b \in B \setminus\{1\}$ fixing a vertex $u$ of $\cT$. But $\st_G(u) \cap \st_G(v)=\st_G([u,v])$ has finite index in $\st_G(u)$, whence $b^n \in \langle a \rangle \setminus\{1\}$, for some $n \in \N$. Thus $b^n \in B\cap \langle a \rangle \neq\{1\}$.    
\end{proof}

We say that a subgroup $B$ of a group $G$ is \emph{weakly malnormal} if there exists $g \in G$ such that the intersection $B \cap gBg^{-1}$ is finite. In view of Lemma~\ref{lem:GBS->non-cyc_sbps_contain_comm_cyclic}, the existence of a non-free weakly malnormal subgroup in $G$ implies that $G$ cannot be a generalised Baumslag--Solitar group. This can be combined with  Theorem~\ref{thm:one-relator-free-or-snormal} to obtain the following corollary.

\begin{cor}\label{cor:A-weakly_malnormal} Suppose that $A$ is a subgroup of a one-relator group $G$ such that $A$ has non-trivial centre. If $G$ contains a non-free weakly malnormal subgroup then there is an infinite order element $f \in G$ such that the subgroup $\langle f ,A \rangle$ is naturally isomorphic to $\langle f \rangle *A \cong \Z*A$ in $G$.    
\end{cor}

\section{Proof of Theorem~\ref{thm:main}} \label{sec:ThmMain}
In this section we prove Theorems~\ref{thm:P_n} and \ref{thm:main}. The proof of the former is nearly complete, modulo the following observation. 

\begin{lemma}\label{lem:inj_on_trace->inj_on_unions} Suppose that $\Gamma$ is a graph, $G$ is a group and $\psi\colon A(\Gamma) \to G$ is a group homomorphism that is injective on the monoid of positive words $A(\Gamma)^+$ in $A(\Gamma)$. If $U,V,W$ are subsets of $V(\Gamma)$ spanning complete subgraphs and $x,y \in A(\Gamma)^+$ then $\psi$
is injective on the union 
\begin{equation*}\label{eq:new_union_of_3}
\mathcal{U}=x\langle U \rangle x^{-1} \cup \langle V\rangle \cup y^{-1} \langle W \rangle y \subseteq A(\Gamma).    
\end{equation*}
\end{lemma}

\begin{proof} Suppose that $\psi(g)=\psi(h)$, for some $g,h \in \mathcal{U}$. We will only consider the case when $g\in x\langle U \rangle x^{-1}$ and $h \in y^{-1} \langle W \rangle y$, as the other cases can be treated similarly.

Since $U$ spans a complete subgraph of $\Gamma$, the subgroup $\langle U \rangle$ is abelian in $A(\Gamma)$, so every element in this subgroup can be written as a product $a b^{-1}$, where $a,b \in \langle U \rangle \cap A(\Gamma)^+$; and similarly for $W$. Therefore, $g=x a b^{-1} x^{-1}$ and $h=y^{-1} d^{-1}c y$, for some $a,b,c,d \in A(\Gamma)^+$. The equation $\psi(g)=\psi(h)$ is therefore equivalent to 
\begin{equation*}\label{eq:re-writing_of_eq}
\psi(dyxa )=\psi(cyxb),
\end{equation*}
which is an equality between $\psi$-images of elements  $dyxa,cyxb \in A(\Gamma)^+$.  But $\psi$ is injective on $A(\Gamma)^+$, whence the arguments of $\psi$ in the re-arranged equation must be the same, which implies that $g=h$ in $A(\Gamma)$. 
Therefore, $\psi$ is injective on $\mathcal{U}$, and the lemma is proved. 
\end{proof}

\begin{proof}[Proof of Theorem~\ref{thm:P_n}] Suppose that $n \in \{3,4\}$, $G$ is a one-relator group and elements $a_1,\dots,a_n \in G$ generate a copy of $T(P_n)$ in $G$ in the natural way. Then there is a homomorphism $\psi\colon A(P_n) \to G$, where $A(P_n)$ is given by the presentation \eqref{eq:pres_HN}, such that $\psi(\alpha_i)=a_i$, for $i=1,\dots,n$, and this homomorphism is injective on the monoid of positive words $A(P_n)^+$.  
Now we can combine Lemma~\ref{lem:inj_on_trace->inj_on_unions} with Propositions~\ref{prop:P3} and \ref{prop:n=4} to conclude that $A(P_n)$ embeds as a subgroup of $\langle a_1,\dots,a_n  \rangle \leqslant G$, as required.
\end{proof}

In order to establish Theorem~\ref{thm:main} we need a few more auxiliary statements. Given a finite graph $\Gamma$, we will use $\d(\Gamma)$ to denote the maximum of the  diameters of connected components of $\Gamma$.

\begin{lemma}\label{lem:tree_of_small_diam} For any $d \in \{0,1,2\}$, if $\Gamma$ is a finite tree with $\d(\Gamma) \le d$ then $A(\Gamma)$ embeds into $A(P_{d+1})$. 
\end{lemma}    

\begin{proof}
This is obvious if $\d(\Gamma) \in \{0,1\}$. In the remaining case, when $\d(\Gamma)=d=2$, $\Gamma$ consists of a single central vertex adjacent to $k \geq 2$ leaves (vertices of degree $1$), implying that $A(\Gamma) \cong F_k \times \Z$, where $F_k$ denotes the free group of rank $k$. Clearly this group embeds into $F_2 \times \Z \cong A(P_3)$, as required.
\end{proof}

It is a result of Kim and Koberda \cite[Theorem~1.8]{KimKoberda}, that $A(P_4)$ contains a subgroup isomorphic to $A(\Gamma)$, for any finite forest $\Gamma$. For disconnected forests of small diameters we will use the following lemma, which is the reason why in the previous two sections we were interested in finding free products $A*\Z$.

\begin{lemma} \label{lem:free-products-embed}
Let $\Gamma$ be a non-empty finite forest with $d=\d(\Gamma) \leq 2$. Then $A(\Gamma)$ embeds into the free product $A(P_{d+1}) * \Z$.    
\end{lemma}

\begin{proof} Let $\Gamma_1,\dots,\Gamma_n$ be the connected components of $\Gamma$, then $\displaystyle A(\Gamma) \cong \mathop{\ast}_{i=1}^n A(\Gamma_i)$.
Moreover, since the diameter of each component is at most $d \le 2$, Lemma~\ref{lem:tree_of_small_diam} tells us that $A(\Gamma_i)$ embeds into $A=A(P_{d+1})$, for every $i=1,\dots,n$. Therefore $A(\Gamma)$ embeds into the free product of $n$ copies of $A$. This $n$-fold free power of $A$ is isomorphic to the free product
\[A*fAf^{-1}*f^2Af^{-2}* \dots *f^{n-1} Af^{-(n-1)},\] which  naturally sits as a subgroup of $A*\langle f \rangle \cong A*\Z$. Hence $A(\Gamma)$ embeds into $A*\Z$.
\end{proof}

We are now ready to prove Theorem~\ref{thm:main}.

\begin{proof}[Proof of Theorem~\ref{thm:main}]
The assumption that $T(\Gamma)$ embeds in $G$ means that there is a homomorphism $\psi\colon A(\Gamma) \to G$ which is injective on the positive monoid $A(\Gamma)^+$.
Our proof will use the maximal diameter $\d(\Gamma)$, of a connected component of $\Gamma$. Note that $\d(\Gamma) \ge 1$ because $\Gamma$ contains at least one edge.

Suppose, first, that $\d(\Gamma)=d \in \{1,2\}$. Then $\Gamma$ has an induced subgraph $\Delta$ isomorphic to $P_{d+1}$, and let $\alpha,\beta$ be any two distinct adjacent vertices of $\Delta$.
The positive submonoid generated by $V=V(\Delta)$ in $A(\Gamma)^+$ is naturally isomorphic to $T(P_{d+1})$, so by Theorem~\ref{thm:P_n}, $A(P_{d+1})$ embeds into $A=\psi(\langle V \rangle)$ in $G$. If $\Gamma$ is connected then $A(\Gamma)$ embeds into $A(P_{d+1})$ by Lemma~\ref{lem:tree_of_small_diam} and we are done. If $\Gamma$ is disconnected then there is a vertex $\lambda \in V(\Gamma)$ that is not adjacent to $\alpha$ or $\beta$ in $\Gamma$, which implies that 
\begin{equation*}\label{eq:alm_maln}
\langle\alpha,\beta\rangle \cap\lambda\langle\alpha,\beta\rangle\lambda^{-1} =\{1\} \text{ in }A(\Gamma).    
\end{equation*}
According to Lemma~\ref{lem:inj_on_trace->inj_on_unions}, $\psi$ is injective on the union $\langle\alpha,\beta\rangle \cup\lambda\langle\alpha,\beta\rangle\lambda^{-1}$ and on the centre of $\langle V \rangle \cong A(P_{d+1})$. Therefore, the subgroup $A=\psi(\langle V \rangle) \leqslant G$ has infinite centre and its subgroup $B=\psi(\langle\alpha,\beta\rangle) \cong \Z^2$ is weakly malnormal  in $G$. So, by Corollary~\ref{cor:A-weakly_malnormal}, the free product $A*\Z$ embeds in $G$, whence $ A(P_{d+1})*\Z$ embeds in $G$. We can now apply Lemma~\ref{lem:free-products-embed} to conclude that $A(\Gamma)$ embeds in $G$.

Finally, suppose $\d(\Gamma) \ge 3$. Then $\Gamma$ has an induced subgraph isomorphic to $P_4$, implying that $T(P_4)$ embeds into $G$. By Theorem~\ref{thm:P_n}, $A(P_4)$ also embeds into $G$. But by a result of Kim and Koberda \cite[Theorem~1.8]{KimKoberda}, $A(P_4)$ has a subgroup isomorphic to $A(\Gamma)$. Hence $A(\Gamma)$ embeds into $G$.
\end{proof}

\section{Some examples and questions}
\label{sec:F2F2F2}
\subsection{Trace submonoids of one-relator groups} In this subsection we collect some examples and open questions that naturally arise in the context of this paper.
\begin{rem}
In Proposition~\ref{prop:P3} we can always take $l=1$. 

Indeed, suppose that $a,b,c \in G$ are elements such that $[a,b]=[b,c]=1$ and $k,l,m \in \Z\setminus\{0\}$ are integers such that $a^k$, $b^l$ and $c^m$ generate a copy of $A(P_3)$ in $G$ in the natural way. Then the subgroup $\langle a^k,c^m \rangle$ must be naturally isomorphic to the free group $F_2$ of rank $2$. Since $b$ is central in $\langle a,b,c \rangle$ and has infinite order, we can conclude that $\langle b \rangle \cap \langle a^k,c^m \rangle=\{1\}$, hence the elements $a^k,b,c^m$ generate a copy of $F_2 \times \Z \cong A(P_3)$ in the natural way.
\end{rem}

However, the following example shows that we cannot further improve Proposition~\ref{prop:P3} to require that $k=l=m=1$.

\begin{ex} \label{ex:P3-need-powers}
Consider the fundamental group of the trefoil knot complement 
\[
G = \langle x,y \mid x^3 = y^2 \rangle,
\]
and let $a = x^2y$, $b = x^3$ and $c = xy$. Then $b$ is central in $G$, and the quotient-group $G/\langle b \rangle$ is naturally isomorphic to the free product $C_3 * C_2$.
Using normal forms it is easy to see that the images of $a$ and $c$ freely generate a free submonoid in $G/\langle b \rangle$. This  implies that $a$, $b$ and $c$ generate $T(P_3)$ in $G$ in the natural way, so the assumptions of Proposition~\ref{prop:P3} are satisfied by Lemma~\ref{lem:inj_on_trace->inj_on_unions}. However, the elements $a,b,c$ do not generate a copy of $A(P_3)$ in $G$ since we have $b = (ac^{-1})^3$, for instance (in fact, $\langle a,b,c \rangle=G$ is not a right-angled Artin group).
\end{ex}

A similar example can be given in the case of Proposition~\ref{prop:n=4}. This improves on \cite[Proposition~6.14]{FGNB}, giving an example of a \emph{one-relator} group $G$ that has four elements generating $T(P_4)$ but not $A(P_4)$.

\begin{ex} \label{ex:P4-need-powers-v2}
Let $G = \langle x,y,t \mid x^3=y^2, tx^3t^{-1}=xy \rangle$, and note that the expression $y = x^{-1}tx^3t^{-1}$ obtained from the second relation results in a one-relator presentation $\langle x,t \mid x^{-3}  (x^{-1}tx^3t^{-1})^2=1 \rangle$ for $G$. Consider the elements $a = x^2y$, $b = x^3$, $c = xy$ and $d = txyt^{-1}$ in $G$. 
It follows from the former presentation that $G$ is an HNN-extension of $G_0 = \langle x,y \mid x^3=y^2 \rangle$; consequently, the elements $a,b,c \in G_0$ do not generate a copy of $A(P_3)$ in $G_0$ (and, hence, in $G$) by Example~\ref{ex:P3-need-powers}, implying that $a,b,c,d$ do not generate a copy of $A(P_4)$ in $G$ in the natural way.

However, we claim that the elements $a,b,c,d$ generate a copy of $T(P_4)$ in the natural way.  Indeed, since $G$ is an HNN-extension of $G_0$ with stable letter $t$ and associated subgroups $\langle b \rangle$ and $\langle c \rangle = t \langle b \rangle t^{-1}$, it follows from Britton's Lemma \cite[p.~181]{LyndonSchupp} that the subgroup $\langle G_0, tG_0t^{-1} \rangle \leqslant G$ splits as an amalgamated free product $G_0 *_{\langle c \rangle = t \langle b \rangle t^{-1}} tG_0t^{-1}$. Since we have $\langle a,b,c \rangle = G_0$ by Example~\ref{ex:P3-need-powers} and since $\langle c,d \rangle = t \langle b,c \rangle t^{-1} \cong \Z^2$ is a subgroup of $tG_0t^{-1}$ containing $t \langle b \rangle t^{-1}$, one sees that the subgroup $H=\langle a,b,c,d \rangle \leqslant G$ splits as an amalgamated free product $H \cong G_0 *_{\langle c \rangle} \langle c,d \rangle$. Thus $H$ has a presentation
\[
\langle X,d \mid Q, dcd^{-1} = c \rangle,
\]
where $X = \{ a,b,c \}$ and $\langle X \mid Q \rangle$ is a presentation of $G_0$ in these generators. By Example~\ref{ex:P3-need-powers}, the submonoid generated by $a,b,c \in G_0$ is naturally isomorphic to $T(P_3)$, so we can apply \cite[Theorem~6.15]{FGNB} to deduce that the elements $a,b,c,d$ generate a copy of $T(P_4)$ in $H$ (and, thus, in $G$) in the natural way, as claimed.
\end{ex}

In Propositions~\ref{prop:P_4-enhanced} and \ref{prop:n=4}, to generate a subgroup isomorphic to $A(P_4)$ in addition to raising the generators $a,b,c,d$, of a copy of $T(P_4)$ in $G$, to sufficiently large powers we also needed to replace $a$ and $d$ with $aca^{-1}$ and $d^{-1} b d$, respectively. This is in contrast with Proposition~\ref{prop:P3}, so it is natural to ask the following.

\begin{question}\label{q:conj_not_required} Suppose that $G$ is a one-relator group and elements $a,b,c,d \in G$ generate a submonoid isomorphic to $T(P_4)$ in the natural way. Do there exist $k,l,m,n \in \Z\setminus\{0\}$ such that $a^k,b^l,c^m,d^n$ generate a copy of $A(P_4)$ in $G$ in the natural way? What if we replace $P_4$ by an arbitrary finite tree $\Gamma$?    
\end{question}

As we have already mentioned in Remark~\ref{rem:RAAG_emb_in_1-rel_iff_forest}, a right-angled Artin group $A(\Gamma)$ embeds in a one-relator group if and only if $\Gamma$ is a forest. The proof of the ``only if'' part in \cite{Gray} uses homological arguments, based on the work of Louder and Wilton \cite{LouWil}, which implies that $A(C_n)$ cannot embed in a one-relator group, where $C_n$ is the cycle of length $n \ge 3$. To fully characterise trace monoids that are embeddable into one-relator groups it remains to answer the following question.

\begin{question}\label{q:T(C_n)_in_one-rel} Can the trace monoid $T(C_n)$, for $n \ge 3$, embed in a one-relator group?    
\end{question}

The answer is negative for $n=3$ because one-relator groups cannot contain subgroups isomorphic to $\Z^3$. For $n=4$ the negative answer follows from a result of Bagherzadeh \cite[Theorem~A]{Bagh}, because $T(C_4)$ splits as the direct product of two  non-abelian free monoids.  For $n=5$ it may be possible to get a negative answer by using methods similar to the proof of Proposition~\ref{prop:n=4}. However, we do not know how to answer Question~\ref{q:T(C_n)_in_one-rel} when $n \ge 6$.

\subsection{Groups containing \texorpdfstring{$T(P_4)$}{T(P\_4)} but not \texorpdfstring{$A(P_4)$}{A(P\_4)}}
This subsection is inspired by the following question that is implicit in \cite[Subsection~6.1]{FGNB}: if a group contains a copy of $T(P_4)$, must it also contain $A(P_4)$ as a subgroup? The motivation for this question stems from the result of Foniqi, Gray and Nyberg-Brodda  \cite[Corollary~6.4]{FGNB} stating that if a group $G$ contains a submonoid isomorphic to $T(P_4)$ then there is a rational subset in $G$ with undecidable membership problem.

To produce groups containing $T(P_4)$ we will use the following lemma.

\begin{lemma} \label{lem:F2F2F2}
The trace monoid $T(P_4)$ embeds into the monoid \[S = M_2 \times M_2 \times M_2, \] where $M_2=\{x,y\}^*$ is the free monoid of rank $2$.
\end{lemma}

\begin{proof}
Write $\alpha=\alpha_1$, $\beta = \alpha_2$, $\gamma = \alpha_3$ and $\delta = \alpha_4$ for the four generators of $T(P_4)$, as in presentation~\eqref{eq:pres_HN}. Given an element $R \in T(P_4)$, we will write $|R|$ for the length of a shortest word in $\{\alpha,\beta,\gamma,\delta\}^*$ representing $R$. We will say that $R$ \emph{starts with} a word $U \in \{\alpha,\beta,\gamma,\delta\}^*$ if $U$ is the prefix of some word of minimal length representing $R$ in $T(P_4)$.

Consider the assignment $\varphi\colon \{\alpha,\beta,\gamma,\delta\} \to S$, defined by
\begin{align*}
\varphi(\alpha) &= (x,y,1), \\
\varphi(\beta)  &= (x,1,x), \\
\varphi(\gamma) &= (1,x,1), \\
\varphi(\delta) &= (y,x,y).
\end{align*}
It is easy to verify that $\varphi$ extends to a (unique) monoid homomorphism $\varphi\colon T(P_4) \to S$. We claim that $\varphi$ is injective.

Suppose, for contradiction, that $\varphi$ is not injective, and let $P,Q \in T(P_4)$ be two distinct elements such that $\varphi(P) = \varphi(Q)$. We may choose such $P$ and $Q$ so that $|P|+|Q|$ is as small as possible. Since $S$ is cancellative (as it embeds in a group),  $P$ and $Q$ cannot start with the same letter in $\{\alpha,\beta,\gamma,\delta\}$ (by the minimality assumption). It is also clear that both $P$ and $Q$ must be non-trivial.

If either $P$ or $Q$ (without loss of generality, $P$) starts with $\alpha$, then, by looking at the projection to the second copy of $M_2$ in $S$, we can deduce that $Q$ starts with $\beta^k \alpha$, for some $k \in \mathbb{N}$. But since $\alpha\beta = \beta\alpha$ in $T(P_4)$, we see that $Q$ starts with $\alpha \beta^k$, contradicting the minimality of $|P|+|Q|$. Thus neither $P$ nor $Q$ can start with $\alpha$; a similar argument (using the projection to the first copy of $M_2$ in $S$) shows that neither $P$ nor $Q$ can start with $\delta$. We may therefore assume (after swapping $P$ and $Q$, if necessary) that $P$ and $Q$ start with $\beta$ and $\gamma$, respectively.

Now let $U$ be the longest word in $\{\alpha,\beta,\gamma\}^*$ such that $Q$ starts with $U$. Since $P$ starts with $\beta$, it follows, by looking at the projection to the third copy of $M_2$ in $S$, that $U$ contains at least one instance of $\beta$. But since $\alpha\beta = \beta\alpha$ and $\gamma\beta = \beta\gamma$ in $T(P_4)$, we can re-write $U$ in $T(P_4)$ as $\beta^k V$, for some $V \in \{\alpha,\gamma\}^*$ and $k \geq 1$. Hence both $P$ and $Q$ start with $\beta$, again contradicting the minimality of $|P|+|Q|$. Thus  $\varphi$ must be injective, as claimed.
\end{proof}

Note that the monoid $S$ from Lemma~\ref{lem:F2F2F2} is also a trace monoid, isomorphic to $T(K_{2,2,2})$, where $K_{2,2,2}$ is the $1$-skeleton of an octahedron. 

We will now give two examples of groups containing copies of $T(P_4)$ as submonoids but not containing subgroups isomorphic to $A(P_4)$.

\begin{ex}
By Lemma~\ref{lem:F2F2F2}, the trace monoid $T(P_4)$ is a submonoid of the direct product $G = F_2 \times F_2 \times F_2$, where $F_2$ is the free group of rank $2$. Note that $G\cong A(K_{2,2,2})$ is a right-angled Artin group, and, as $K_{2,2,2}$ has no induced subgraphs isomorphic to $P_4$, we know, by \cite[Theorem~1.7]{KimKoberda}, that $A(P_4)$ does not embed into $G$. 

This gives an easy example of a residually free group $G$ that has a submonoid isomorphic to $T(P_4)$. On the other hand, it is well-known that $A(P_4)$ is not residually free, so it cannot embed into a residually free group.
\end{ex}

\begin{ex} As we have already seen in Example~\ref{ex:BS}, the free monoid $M_2$ embeds into the Baumslag--Solitar group $BS(1,2)$. Hence, by Lemma~\ref{lem:F2F2F2}, $T(P_4)$ embeds into the direct product $H$ of three copies of $BS(1,2)$. Now, $H$ is a metabelian group, so it does not contain non-abelian free subgroups. In particular, $A(P_4)$ is not a subgroup of $H$.    
\end{ex}

\section{\texorpdfstring{$C^*$}{C*}-simplicity of one-relator groups}\label{sec:C*-simple}
In this section we characterise the $C^*$-simplicity of one-relator groups.
Since any group with property $P_{nai}$ is $C^*$-simple, in view of Corollary~\ref{cor:P_nai}, we just need to investigate the $C^*$-simplicity of generalised Baumslag--Solitar groups.

Recall that a generalised Baumslag--Solitar group $G$ is said to be \emph{unimodular} if for any $x,y \in G\setminus\{1\}$, the equation $xy^m x^{-1}=y^n$ implies that $|m|=|n|$. This condition is equivalent to the requirement that the image of any \emph{modular homomorphism} $G \to \mathbb{Q}^*$ is contained in $\{-1,1\}$ (see \cite[Section~2]{Levitt}). By \cite[Proposition~2.6]{Levitt} a  generalised Baumslag--Solitar group is unimodular if and only if it  has an infinite cyclic normal subgroup; in this case it also contains a finite index subgroup splitting as a direct product of a free group with $\mathbb{Z}$.

The following statement combines results of de la Harpe and Pr\'{e}aux \cite{dlH-Pre} and of Brownlowe, Mundey, Pask, Spielberg and Thomas \cite{BMPST} with well-known properties of graphs of groups.

\begin{prop}\label{prop:GBS-C*-simple} Let $G$ be a non-trivial generalised Baumslag--Solitar group. If $G$ is not $C^*$-simple then either $G$ is isomorphic to a solvable Baumslag--Solitar group $BS(1,n)$, for some $n \in \Z\setminus\{0\}$, or $G$ is unimodular.
\end{prop}

\begin{proof} By the assumptions, $G=\pi_1(\mathcal{G},\Gamma)$, where $(\mathcal{G},\Gamma)$ is a non-empty finite graph of groups with infinite cyclic vertex and edge groups. Moreover, after collapsing finitely many \emph{trivial edges}, we can suppose that this graph of groups is \emph{reduced}, i.e., for every edge $e \in E(\Gamma)$, starting at a vertex $u$ and ending at a vertex $v$, if $u \neq v$ then the embeddings of $G_e$ in $G_u$ and $G_v$ are proper (see \cite[Section~4.1]{dlH-Pre}).

Let $\cT$ be the Bass--Serre for the given splitting of $G$ as the fundamental group of $(\mathcal{G},\Gamma)$. Then $G$ acts on $\cT$ with infinite cyclic vertex and edge stabilisers.
Since $(\mathcal{G},\Gamma)$ is reduced, the action of $G$ on $\cT$ is \emph{minimal}, i.e., there is no proper invariant subtree \cite[Proposition~7.12]{Bass}. If $\Gamma$ consists of a single vertex $v$ and no edges then $G=G_v \cong\Z$ and so it is unimodular. Thus we can further suppose that $\Gamma$ has at least one edge. 

If the action of $G$ fixes an end in $\partial\cT$ then $V(\Gamma)=\{v\}$ and $G$ is isomorphic to an ascending HNN-extension of $G_v \cong \Z$ by \cite[Proposition~4.13]{MinasyanOsin}. In other words, $G \cong BS(1,n)$, for some $n \in \Z\setminus\{0\}$.

If $\cT$ is a simplicial line then the kernel $N$ of the action of $G$ on this line is a cyclic normal subgroup of $G$. If $N=\{1\}$ then $G$ is a subgroup of the infinite dihedral group $\mathrm{Aut}(\cT)$, hence $G$ must be infinite cyclic as it is torsion-free. Otherwise, $N \cong \Z$. In either case we conclude that $G$ is unimodular.

Thus we can suppose that $G$ does not fix any end of $\cT$ and $\cT$ is not a line. Then the $G$-action on $\cT$ is \emph{strongly hyperbolic}, see \cite[Proposition~14]{dlH-Pre}. And if $G$ is not unimodular then the action of $G$ on $\partial \cT$ is \emph{topologically free} by \cite[Corollary~7.11]{BMPST}. In the terminology of \cite{dlH-Pre} the latter means that the action of $G$ on $\cT$ is \emph{slender}, and we can apply \cite[Corollaries~15 and 2]{dlH-Pre} to deduce that $G$ is $C^*$-simple.
\end{proof}

\begin{proof}[Proof of Corollary~\ref{cor:C*-simple}] Let $G$ be a  group given by presentation \eqref{eq:1-rel_pres} with $k \ge 2$. 
If $G$ satisfies \ref{it:C*-BS1n} then it is infinite and solvable and if $G$ satisfies \ref{it:C*-unimod} then it contains an infinite cyclic normal subgroup. In either case $G$ contains a non-trivial amenable normal subgroup, so it cannot be $C^*$-simple (see, for example, \cite[Proposition~3]{dlH-survey}).

For the opposite direction, suppose that $G$ is not $C^*$-simple. The work of Bekka, Cowling and de la Harpe \cite[Lemmas~2.2 and 2.1]{Bek-Cow-Har} shows that $G$ does not have property $P_{nai}$. Using Theorem~\ref{thm:one-relator-free-or-snormal} and Remark~\ref{rem:at_least_3_gens} we can deduce that 
$k=2$ and $G$ must be a generalised Baumslag--Solitar group. We can now apply
Proposition~\ref{prop:GBS-C*-simple} to conclude that $G$ must satisfy \ref{it:C*-BS1n} or \ref{it:C*-unimod}, as required.    
\end{proof}

\begin{rem}\label{rem:deciding_if_C*-simpl} There is an algorithm taking on input presentation \eqref{eq:1-rel_pres} and deciding whether or not $G$ satisfies conditions \ref{it:C*-BS1n} or \ref{it:C*-unimod} from Corollary~\ref{cor:C*-simple}. 

Indeed, this can be done by analysing the proofs of Proposition~\ref{prop:1-rel_ah-refinement} and \cite[Proposition~4.21]{MinasyanOsin}. More precisely, it is possible to decide if $G$ is in case \ref{it:MO-snormal} from the proof of Proposition~\ref{prop:1-rel_ah-refinement} by using Howie's algorithm for computing intersection of Magnus subgroups in a one-relator group \cite[Theorem~E]{Howie}. It would also tell us whether $a^n=b^{\pm n}$, for some $n \in \Z\setminus\{0\}$, which happens if and only if $G$ is unimodular. We can also decide whether $G$ is in case \ref{it:MO-mappingtorus} by using the solution to the membership problem for Magnus subgroups of one-relator groups \cite[Theorem~IV.5.3]{LyndonSchupp}. Moreover, if $G$ is indeed an ascending HNN-extension of a free group $F$, we can compute the rank of $F$ and check if the corresponding endomorphism of $F$ is an automorphism; if true then we can determine whether this automorphism has finite order in $\mathrm{Out}(F)$ because we know the bounds for the orders of finite subgroups in $\mathrm{Out}(F)$ \cite[Theorem on p.~83]{Wan-Zim}.
\end{rem}

\bibliographystyle{amsplain}
\bibliography{ref}

\providecommand{\bysame}{\leavevmode\hbox to3em{\hrulefill}\thinspace}
\providecommand{\MR}{\relax\ifhmode\unskip\space\fi MR }
\providecommand{\MRhref}[2]{%
  \href{http://www.ams.org/mathscinet-getitem?mr=#1}{#2}
}
\providecommand{\href}[2]{#2}
\begin{thebibliography}{10}

\bibitem{AbbottDahmani}
C.~R. Abbott and F.~Dahmani, \emph{Property ${P}_{naive}$ for acylindrically
  hyperbolic groups}, Math. Z. \textbf{291} (2019), 555--568.

\bibitem{Bagh}
G.~H. Bagherzadeh, \emph{Commutativity in one-relator groups}, J. Lond. Math.
  Soc. \textbf{13} (1976), no.~3, 459--471.

\bibitem{Bass}
H.~Bass, \emph{Covering theory for graphs of groups}, J. Pure Appl. Algebra
  \textbf{89} (1993), no.~1-2, 3--47.

\bibitem{Bek-Cow-Har}
M.~Bekka, M.~Cowling, and P.~de~la Harpe, \emph{Some groups whose reduced
  {$C^*$}-algebra is simple}, Inst. Hautes \'{E}tudes Sci. Publ. Math. (1994),
  no.~80, 117--134 (1995).

\bibitem{Bieri}
R.~Bieri, \emph{Homological dimension of discrete groups}, second ed., Queen
  Mary College Mathematics Notes, Queen Mary College, Department of Pure
  Mathematics, London, 1981.

\bibitem{BKKO}
E.~Breuillard, M.~Kalantar, M.~Kennedy, and N.~Ozawa, \emph{{$C^*$}-simplicity
  and the unique trace property for discrete groups}, Publ. Math. Inst. Hautes
  \'{E}tudes Sci. \textbf{126} (2017), 35--71.

\bibitem{BridsonHaefliger}
M.~R. Bridson and A.~Haefliger, \emph{Metric spaces of non-positive curvature},
  Grundlehren der mathematischen Wissenschaften, Springer-Verlag, 1999.

\bibitem{BMPST}
N.~Brownlowe, A.~Mundey, D.~Pask, J.~Spielberg, and A.~Thomas,
  \emph{{$C^*$}-algebras associated to graphs of groups}, Adv. Math.
  \textbf{316} (2017), 114--186.

\bibitem{Bu-Kr}
J.~O. Button and R.~P. Kropholler, \emph{Nonhyperbolic free-by-cyclic and
  one-relator groups}, New York J. Math. \textbf{22} (2016), 755--774.

\bibitem{Charney}
R.~Charney, \emph{An introduction to right-angled {Artin} groups}, Geom.
  Dedicata \textbf{125} (2007), 141--158.

\bibitem{dlH-survey}
P.~de~la Harpe, \emph{On simplicity of reduced {$C^\ast$}-algebras of groups},
  Bull. Lond. Math. Soc. \textbf{39} (2007), no.~1, 1--26.

\bibitem{dlH-Pre}
P.~de~la Harpe and J.-P. Pr\'{e}aux, \emph{{$C^*$}-simple groups: amalgamated
  free products, {HNN} extensions, and fundamental groups of 3-manifolds}, J.
  Topol. Anal. \textbf{3} (2011), no.~4, 451--489.

\bibitem{DD}
W.~Dicks and M.~J. Dunwoody, \emph{Groups acting on graphs}, Cambridge Studies
  in Advanced Mathematics, vol.~17, Cambridge University Press, Cambridge,
  1989.

\bibitem{DicksLeary}
W.~Dicks and I.~J. Leary, \emph{Presentations for subgroups of {Artin} groups},
  Proc. Amer. Math. Soc. \textbf{127} (1999), no.~2, 343--348.

\bibitem{Diekert}
V.~Diekert, \emph{Combinatorics on traces}, Lecture Notes in Computer Science,
  no. 454, Springer-Verlag, 1990.

\bibitem{FGNB}
I.~Foniqi, R.~D. Gray, and C.-F. Nyberg-Brodda, \emph{Membership problems for
  positive one-relator groups and one-relation monoids}, preprint,
  \href{https://arxiv.org/abs/2305.15672}{arXiv:2305.15672} [math.GR], 2023.

\bibitem{GenevoisHorbez}
A.~Genevois and C.~Horbez, \emph{Acylindrical hyperbolicity of automorphism
  groups of infinitely ended groups}, J. Topol. \textbf{14} (2021), no.~3,
  963--991.

\bibitem{Gray}
R.~D. Gray, \emph{Undecidability of the word problem for one-relator inverse
  monoids via right-angled {Artin} subgroups of one-relator groups}, Invent.
  Math. \textbf{219} (2020), 987--1008.

\bibitem{Green}
E.~R. Green, \emph{Graph products of groups}, Ph.D. Thesis, University of
  Leeds, 1990.

\bibitem{Howie}
J.~Howie, \emph{Magnus intersections in one-relator products}, Michigan Math.
  J. \textbf{53} (2005), no.~3, 597--623.

\bibitem{Kal-Ken}
M.~Kalantar and M.~Kennedy, \emph{Boundaries of reduced {$C^*$}-algebras of
  discrete groups}, J. Reine Angew. Math. \textbf{727} (2017), 247--267.

\bibitem{KimKoberda}
S.~Kim and T.~Koberda, \emph{Embedability between right-angled {Artin} groups},
  Geom. Topol. \textbf{17} (2013), no.~1, 493--530.

\bibitem{PKr}
P.~H. Kropholler, \emph{Baumslag-{S}olitar groups and some other groups of
  cohomological dimension two}, Comment. Math. Helv. \textbf{65} (1990), no.~4,
  547--558.

\bibitem{Levitt}
G.~Levitt, \emph{On the automorphism group of generalized {B}aumslag-{S}olitar
  groups}, Geom. Topol. \textbf{11} (2007), 473--515.

\bibitem{Linton}
M.~Linton, \emph{One-relator hierarchies}, preprint,
  \href{https://arxiv.org/abs/2202.11324}{arXiv:2202.11324} [math.GR], 2022.

\bibitem{LohSte}
M.~Lohrey and B.~Steinberg, \emph{The submonoid and rational subset membership
  problems for graph groups}, J. Algebra \textbf{320} (2008), no.~2, 728--755.

\bibitem{LouWil}
L.~Louder and H.~Wilton, \emph{Stackings and the {$W$}-cycles conjecture},
  Canad. Math. Bull. \textbf{60} (2017), no.~3, 604--612.

\bibitem{Lyndon}
R.~C. Lyndon, \emph{Cohomology theory of groups with a single defining
  relation}, Ann. of Math. (2) \textbf{52} (1950), 650--665.

\bibitem{LyndonSchupp}
R.~C. Lyndon and P.~E. Schupp, \emph{Combinatorial group theory}, Classics in
  mathematics, no.~89, Springer-Verlag, 1977.

\bibitem{Masters}
J.~D. Masters, \emph{Heegaard splittings and $1$-relator groups}, preprint,
  2006.

\bibitem{Min-some_props}
A.~Minasyan, \emph{Some properties of subsets of hyperbolic groups}, Comm.
  Algebra \textbf{33} (2005), no.~3, 909--935.

\bibitem{Min-res_props}
\bysame, \emph{On residual properties of word hyperbolic groups}, J. Group
  Theory \textbf{9} (2006), no.~5, 695--714.

\bibitem{MinasyanOsin}
A.~Minasyan and D.~V. Osin, \emph{Acylindrical hyperbolicity of groups acting
  on trees}, Math. Ann. \textbf{362} (2015), 1055--1105.

\bibitem{Newman-thesis}
B.~B. Newman, \emph{Some aspects of one-relator groups}, Ph.D. Thesis,
  University College of Townsville, 1968.

\bibitem{O-S}
A.~Olijnyk and V.~Sushchansky, \emph{Representations of free products by
  infinite unitriangular matrices over finite fields.}, Internat. J. Algebra
  Comput. \textbf{14} (2004), no.~5-6, 741--749.

\bibitem{Olsh}
A.~Yu. Olshanskii, \emph{On residualing homomorphisms and {$G$}-subgroups of
  hyperbolic groups}, Internat. J. Algebra Comput. \textbf{3} (1993), no.~4,
  365--409.

\bibitem{Osin}
D.~V. Osin, \emph{Acylindrically hyperbolic groups}, Trans. Amer. Math. Soc.
  \textbf{368} (2016), 851--888.

\bibitem{Paris}
L.~Paris, \emph{Artin monoids inject in their groups}, Comment. Math. Helv.
  \textbf{77} (2002), 609--637.

\bibitem{Ree-Men}
R.~Ree and N.~S. Mendelsohn, \emph{Free subgroups of groups with a single
  defining relation}, Arch. Math. (Basel) \textbf{19} (1968), 577--580.

\bibitem{Rosen}
J.~M. Rosenblatt, \emph{Invariant measures and growth conditions}, Trans. Amer.
  Math. Soc. \textbf{193} (1974), 33--53.

\bibitem{Servatius}
H.~Servatius, \emph{Automorphisms of graph groups}, J. Algebra \textbf{126}
  (1989), no.~1, 34--60.

\bibitem{Short}
H.~Short, \emph{Quasiconvexity and a theorem of {H}owson's}, Group theory from
  a geometrical viewpoint ({T}rieste, 1990), World Sci. Publ., River Edge, NJ,
  1991, pp.~168--176.

\bibitem{Wan-Zim}
S.~C. Wang and B.~Zimmermann, \emph{The maximum order of finite groups of outer
  automorphisms of free groups}, Math. Z. \textbf{216} (1994), no.~1, 83--87.

\bibitem{Wise-book}
D.~T. Wise, \emph{The structure of groups with a quasiconvex hierarchy}, Annals
  of Mathematics Studies, vol. 209, Princeton University Press, Princeton, NJ,
  2021.

\end{thebibliography}
\end{document}